\documentclass[12pt]{amsart}

\usepackage{amsmath, amssymb, amscd, verbatim, xspace,amsthm}
\usepackage{latexsym, epsfig, color}
\usepackage[hidelinks]{hyperref}

\usepackage[OT2,T1]{fontenc}
\DeclareSymbolFont{cyrletters}{OT2}{wncyr}{m}{n}
\DeclareMathSymbol{\Sha}{\mathalpha}{cyrletters}{"58}

\newcommand{\ba}{\begin{align*}}
\newcommand{\ea}{\end{align*}}

\newcommand{\C}{\ensuremath{{\mathbb{C}}}}
\newcommand{\Z}{\ensuremath{{\mathbb{Z}}}\xspace}

\newcommand{\Q}{\ensuremath{{\mathbb{Q}}}}
\newcommand{\R}{\ensuremath{{\mathbb{R}}}}
\newcommand{\F}{\ensuremath{{\mathbb{F}}}}

\newcommand{\E}{\ensuremath{{\mathbb{E}}}}

\newcommand{\ra}{\rightarrow}

\newcommand\Hom{\operatorname{Hom}}
\newcommand\coker{\operatorname{coker}}

\newcommand\Aut{\operatorname{Aut}}

\newcommand\Gal{\operatorname{Gal}}
\newcommand\Nm{\operatorname{Nm}}
\newcommand\Sym{\operatorname{Sym}}

\newcommand\Sur{\operatorname{Sur}}

\newcommand\Ind{\operatorname{Ind}}

\newcommand\tensor{\otimes}
\newcommand\isom{\simeq}
\newcommand\sub{\subset}

\newcommand\Disc{\operatorname{Disc}}
\newcommand\GL{\operatorname{GL}}

\newcommand\cok{\operatorname{cok}}

\newcommand\End{\operatorname{End}}

\renewcommand\O{\mathcal{O}}

\newcommand\Qf{\ensuremath{Q_f}\xspace}

\newcommand\bq{\begin{equation}}
\newcommand\eq{\end{equation}}

\numberwithin{equation}{section}
\newtheorem{proposition}[equation]{Proposition}
\newtheorem{theorem}[equation]{Theorem}
\newtheorem{corollary}[equation]{Corollary}

\newtheorem{lemma}[equation]{Lemma}

\newtheorem{conjecture}[equation]{Conjecture}

\theoremstyle{remark}
\newtheorem{remark}[equation]{Remark}

\usepackage{fullpage,pdfsync}

\newtheorem{definition}[equation]{Definition}

\newcommand\Cl{{\operatorname{Cl}}}
\newcommand\SSur{{\operatorname{Sur}}}
\newcommand\rk{{\operatorname{rk}}}
\newcommand\Sp{{\operatorname{Sp}}}
\newcommand\trye{{e}}
\newcommand\tryf{{f}}
\newcommand\N{{\mathbb{N}}}
\newcommand\Pf{E}
\renewcommand\Qf{F}

\begin{document}

\begin{abstract}
Cohen, Lenstra, and Martinet have given conjectures for the distribution of class groups of extensions of number fields, but Achter and Malle have given theoretical and numerical evidence that these conjectures are wrong regarding the Sylow $p$-subgroups of the class group when the base number field contains $p$th roots of unity.
We give complete conjectures of the distribution of Sylow $p$-subgroups of class groups of extensions of a number field when $p$ does not divide the degree of the Galois closure of the extension.  These conjectures are based on $q\ra\infty$ theorems on these distributions in the function field analog  and use recent work of the authors on explicitly giving a distribution of modules  from its moments.  Our conjecture matches many, but not all, of the previous conjectures that were made in special cases taking into account roots of unity.
\end{abstract}

\title{Conjectures for distributions of class groups of extensions of number fields containing roots of unity}

\author{Will Sawin}
\address{Department of Mathematics\\
Columbia University \\
2990 Broadway \\
New York, NY 10027 USA}  
\email{sawin@math.columbia.edu}

\author{Melanie Matchett Wood}
\address{Department of Mathematics\\
Harvard University\\
Science Center Room 325\\
1 Oxford Street\\
Cambridge, MA 02138 USA}  
\email{mmwood@math.harvard.edu}

\maketitle

\section{Introduction}

In 1984,  Cohen and Lenstra \cite{Cohen1984} gave conjectures for the distribution of the odd parts of class groups of imaginary and real quadratic fields,
as well as for any finite abelian group $A$,  the prime-to-$|A|$ part of class groups of totally real $A$-fields.  Cohen and Martinet \cite{Cohen1990} generalized these conjectures
to the situation of an arbitrary number field $K_0$ as a base field, and arbitrary group $\Gamma$, 
giving conjectures for distributions of the ``good part'' of class groups of $\Gamma$-extensions of a fixed $K_0$ with any fixed behavior at the infinite places of $K_0$.  In particular, the ``good part'' includes the product of the  Sylow $p$-subgroups of the class group for $p\nmid |\Gamma|$.  However, these conjectures appear to be wrong at primes dividing the number of roots of unity in the base field, as shown by Achter  \cite{Achter2006} and Malle \cite{Malle2008} and in some cases can even be disproven using reflection principles (\S\ref{ss-reflection}).

In this paper, we give complete conjectures for the distribution of Sylow $p$-subgroups of class groups of $\Gamma$-extensions
(for $p\nmid |\Gamma|$)  of any number field $K_0$ that contains the $p$th (and possibly further) roots of unity.
\begin{conjecture}\label{C}
Let $\Gamma$ be a finite group and $p$ a prime $p\nmid |\Gamma|$.  Let $S=\Z_p[\Gamma]/(\sum_{\gamma\in \Gamma}\gamma)$.
Let $K_0$ be a number field with $u$ infinite places containing the $p^r$th roots of unity but not the $p^{r+1}$th roots of unity, for some $r\geq 1$. 
Let $\mathcal{E}=\mathcal{E}(\Gamma,K_0)$ be the set of isomorphism classes of  Galois $\Gamma$-extensions $K/K_0$ along with an isomorphism $\Gal(K/K_0)\isom \Gamma$. 
Let $\Cl_{K|K_0}:=\Cl_{K}/\Cl_{K_0}$.

As $K$ varies over $\mathcal{E}$, 
the distribution of $S$-modules $\Cl_{K|K_0}[p^\infty]$ is the one with the average number of surjective
morphisms from $\Cl_{K|K_0}[p^\infty]$ to $V$ being
$$
\frac{|(\wedge^2_{\Z_p}V)^\Gamma[p^r]|}{|V|^u}
$$
for any finite $S$-module $V$.
Theorem~\ref{T:formula}  shows there is a unique distribution on $S$-modules with these moments and gives an explicit formula for it.
\end{conjecture}

Conjecture~\ref{C} is motivated by Theorem~\ref{T:MainFF}, which is based 
on work of Liu, Zureick-Brown, and the second author \cite{Liu2019} and gives the moments of these distributions in a $q\ra\infty$ limit in the function field case.
The explicit formulas obtained in Theorem~\ref{T:formula} are based on recent work of the current authors \cite{Sawin2022} that allows one to explicitly describe a distribution of modules (or more general objects) given its moments (see Theorem~\ref{T:mom}). The formulas for the conjectural distributions, like all previous such formulas, are given in terms of $q$-series.  Some of the formulas are quite involved, and it would be interesting if they could be further simplified with ideas from the study of such $q$-series.    We give now, as a special case, the formulas for just the distribution on $p$-torsion that gives the moments of Conjecture~\ref{C}.

\begin{theorem}\label{T:ptors}
Let $\Gamma$ be a finite group and $p$ a prime such that $p\nmid |\Gamma|$.  Let
$V_1,\dots,V_c$ be the non-trivial irreducible representations of $\Gamma$ over $\F_p$.
Let $\kappa_i=\End_\Gamma(V_i)$ and $q_i=|\kappa_i|$
and 
$\dim_{\kappa_i} V_i=n_i$.
Let $\epsilon_i=-1$ if $(\wedge_{\kappa_i}^2 V_i)^\Gamma\ne 0$, let 
$\epsilon_i=1$ if $(\wedge_{\kappa_i}^2 V_i)^\Gamma= 0$ but $(V_i\tensor_{\kappa_i}V_i)^\Gamma\ne 0$, and let  $\epsilon_i=0$ if $(V_i\tensor_{\kappa_i}V_i)^\Gamma= 0$.
Let $u$ be a positive integer.
Let $R_\Gamma$ be the set
of isomorphism classes of finite dimensional representations $V$ of $\Gamma$ over $\F_p$ with $V^{\Gamma}=0$ (with a trivial topology and $\sigma$-algebra).

Then there is a unique measure $\nu$ on $R_\Gamma$
such that
$$
\int_{X\in R_\Gamma} |\Sur(X,V)|d\nu=\frac{|(\wedge^2_{\F_p}V)^\Gamma|}{|V|^u}
$$
for every  $V\in R_\Gamma$.  
For all non-negative integers $f_1,\dots,f_c$, we have, for $V=\prod_i V_i^{f_i}$,
\begin{align*}
&\nu(\{V\})=\\
&
\frac{|(\wedge^2_{\F_p} V)^\Gamma|}{|V|^u|\Aut(V)|}
\prod_{\substack{V_i\\\textrm{self-dual}}} 
\prod_{\ell\geq 0}(1+ q_i^{-un_i-\frac{\epsilon_i+1}{2}-\ell})^{-1} 
\prod_{\substack{i< j\\ \textrm{$V_i,V_j$ dual}}}  
\frac{\prod_{\ell=1}^\infty(1-q_i^{-\ell})\prod_{\ell=1}^{2un_i}(1-q_i^{-\ell})
}{
\prod_{\ell=1}^{un_i-f_i+f_j}(1-q_i^{-\ell})\prod_{\ell=1}^{un_i+f_i-f_j}(1-q_i^{-\ell})}
.
\end{align*}
if  
 $|f_i -f_j|\leq un_i$ for all  pairs of dual representations $V_i,V_j$,
 and
$ \nu(\{V\})=0$ otherwise.  
\end{theorem}

For $V=\prod_i V_i^{f_i}$, we have $|V|=\prod_i q_i^{f_in_i}$, and $|\Aut(V)|=\prod_i q_i^{f_i^2}\prod_{\ell=1}^{f_i}(1-q_i^{-\ell})$
and 
$$|(\wedge^2_{\F_p} V)^\Gamma|=\prod_{V_i \textrm{ self-dual}} q_i^{\binom{f_i}{2} +(1-\epsilon_i)\frac{f_i}{2}}
\prod_{i<j, V_i,V_j \textrm{ dual}} q_i^{f_if_j}.$$
%

In the setting of 
Conjecture~\ref{C}, Conjecture~\ref{C} and Theorem~\ref{T:ptors} together imply that $\Cl_{K|K_0}[p]$, as a representation of $\Gamma$ over $\F_p$, is distributed as $\nu$ from Theorem~\ref{T:ptors}.

\subsection{Reflection principles and interpretation}\label{ss-reflection}

Perhaps the most important feature of the distribution of Theorem~\ref{T:ptors} is that not all representations $V \in R_\Gamma$ appear with positive probability. Instead, the probability of $\prod_i V_i^{f_i}$ occurring is zero unless $|f_i -f_j|\leq un_i$ for all  pairs of dual representations $V_i,V_j$. This prediction is closely related to the classical number-theoretic phenomenon of reflection principles.

The first reflection principle, due to Scholz, states that for $D <-3 $ a fundamental discriminant, the $3$-ranks of the class groups of $\mathbb Q( \sqrt{D})$ and $\mathbb Q(\sqrt{-3D})$ differ by at most one. To understand generalizations, it is helpful to observe that the $3$-parts of both of these class groups embed into the class group of $\mathbb Q( \sqrt{D}, \sqrt{-3})$. The class group of  $\mathbb Q( \sqrt{D}, \sqrt{-3})$ admits an action by the Galois group $(\mathbb Z/2)^2$ of  $\mathbb Q( \sqrt{D}, \sqrt{-3})$. The $3$-part of the class group splits into eigenspaces of the action corresponding to the characters of the Galois group, of which one is trivial and the other three are naturally the $3$-parts of the class groups of $\mathbb Q(\sqrt{D}), \mathbb Q(\sqrt{-3D})$, and $\mathbb Q(\sqrt{-3})$ (the last one also being trivial). Hence the Scholz reflection principle can be equivalently stated as the claim that two eigenspaces of the Galois action on the class group of $\mathbb Q( \sqrt{D}, \sqrt{-3})$ have ranks that differ by at most $1$.

Leopoldt~\cite{Leopoldt1958} generalized Scholz's reflection principle by considering a Galois extension $K$ of the rational numbers, with Galois group $G$, that contains the $p$th root of unity. The action on the $p$th roots of unity then gives a one-dimensional representation $\chi$ of $G$ over $\mathbb F_p$. For $V_1,\dots, V_c$ the irreducible representations of $G$ over $\mathbb F_p$, we can write the $p$-part of the class group as $\prod_i V_i^{f_i}$. Leopoldt's reflection principle bounds $|f_j-f_i|$ when $ V_j \cong \chi \otimes  V_i^\vee $.

In our setting, we take $K$ an extension of a fixed field $K_0$ with Galois group $\Gamma$. Since we are interested in random extensions, and we expect that, if $K_0$ does not contain the $p$th roots of unity, then almost all extensions $K$ of $K_0$ will not contain the $p$th roots of unity either, the reflection principle will only be relevant when $K_0$ contains the $p$th roots of unity. In this case, since the action of $\Gamma$ on the $p$th roots of unity is trivial, a generalization due to Gras~\cite{Gras1998} of Leopoldt's reflection principle implies that $|f_j-f_i|\leq u n_i$ if $V_j \cong V_i^\vee$.

The importance of reflection principles to class group statistics was pointed out by Breen, Varma, and Voight~\cite{Breen2021}. A conjecture on the distribution of class groups should certainly assign probability zero to every representation which is forbidden by the known reflection principles. Their work proves a reflection principle for the 2-part of the class group of an odd degree abelian Galois extension of $\mathbb Q$, and uses it to make conjectures about the distribution of the class group of such extensions, together with other invariants such as the unit signature and the narrow class group.

Theorem~\ref{T:ptors} implies that Conjecture~\ref{C} is compatible with Gras's reflection principle: Every representation of $\Gamma$ on the $p$-torsion part of the class group which is forbidden by the reflection principle is given probability zero in the predicted distribution. (The converse is also true: every representation given probability zero is explained by a reflection principle).

This is not because we used our knowledge of reflection principles to formulate the conjecture. Instead, Conjecture~\ref{C} is motivated by moment calculations in the function field context. Theoretically, we could have guessed the reflection principle from these moment calculations and Theorem~\ref{T:ptors}. In fact, going beyond the $p$-torsion part described in Theorem~\ref{T:ptors}, we see that Conjecture~\ref{C} implies a relationship between the $p^r$-torsion ranks associated to dual pairs of irreducible representations whenever the base field contains the $p^r$th roots of unity. We prove this relationship -- another reflection principle -- in Section \ref{S:dontappear}.

On the other hand, the conjectures of Cohen and Lenstra \cite{Cohen1984} and  Cohen and Martinet \cite{Cohen1990} state that each (invariant-free) representation of the Galois group $\Gamma$ should occur with positive probability. Thus they contradict the reflection principles in cases where $\Gamma$ has two dual but non-isomorphic irreducible representations over $\mathbb F_p$, the simplest of which occurs when $K_0=\Q$, $\Gamma=C_7$, and $p=2$. The group $C_7$ has two three dimensional irreducible representations $V_1,V_2$ and the representation $V_1^2$, say, can never occur as the $2$-part of the class group of a cyclic degree $7$ field by the reflection principle.

Closely related to reflection principles is the statistical dependence between different representations. The reflection principle $|f_i -f_j| \leq u n_i$, together with the phenomenon familiar from every case of Cohen-Lenstra that the $p$-rank $f_i$ can grow arbitrarily large, imply that $f_i$ and $f_j$ cannot be statistically independent of each other. This is in contrast to previous conjectures which predicted that the parts of  $\Cl_{K|K_0}[p^\infty]$ corresponding to different irreducible representations of $\Gamma$ should be independent. However, it is not clear from the considerations what the exact nature of the dependence should be.  Conjecture~\ref{C} implies independence between all the parts of $\Cl_{K|K_0}[p^\infty]$ corresponding to different irreducible representations except for a representation and its dual, and can be used to give an explicit formula for the dependence between those two representations. Both of these can be seen already for the $p$-parts in Theorem~\ref{T:ptors}.

In Section~\ref{S:example}, we give more details on the example of $K_0=\Q$ and $\Gamma=C_7$, which is the smallest case where we have two dual but non-isomorphic irreducible representations, so that the reflection theorem is relevant and this dependence occurs.

\subsection{Previous work}
In 2006, Achter  \cite{Achter2006} found that in a large $q$ limit, class groups of quadratic extensions of function fields $\F_q(t)$ did not always satisfy the analogues of the Cohen-Martinet conjectures for function fields.
In 2008, Malle \cite{Malle2008} found computational evidence that the conjectures of Cohen-Lenstra-Martinet were incorrect for the Sylow $p$-subgroups of class groups of number fields when $p$ divides the order of the roots of unity in $K_0$.   Garton \cite{Garton2015} showed that the discrepancies Achter found in the function field
case could also be explained by the presence of roots of unity in the base field.  

Since 2008, there have been several papers aimed at correcting these conjectures in the presence of roots of unity.  
In his original paper, Malle \cite{Malle2008} gave suggested conjectures for $K_0=\Q$  and $\Gamma=\Z/3\Z,\Z/5\Z$, as well as $K_0=\Q(\sqrt{-3})$ and $\Gamma=\Z/2\Z$.  
Malle gave further conjectures in a follow-up paper \cite{Malle2010}, including conjectures for $\Gamma=\Z/2\Z$ with the Sylow $p$-subgroup of the class group for any base field containing $p$th but not $p^2$th roots of unity.
In the $\Gamma=\Z/2\Z$ case, Garton \cite{Garton2015} studied the distribution Achter saw arising in the function field large $q$ limit in terms of a random matrix model, and gave the moments of this distribution.  He was able to give explicit formulas for the distribution for base fields containing 
$p^2$th but not $p^3$th  roots of unity.
Adam and Malle \cite{Adam2015} made conjectures in further scenarios, including that for $\Gamma=\Z/2\Z$ with any roots of unity in the base field
the class group distribution should match a limit of certain random matrix distributions, but they did not give explicit formulas for this distribution for a general base field.
Lipnowski, Sawin, and Tsimerman \cite{Lipnowski2020}, for the $\Gamma=\Z/2\Z$ and $K_0=\F_q(t)$ case, considered additional pairing data that arises on class groups in the presence of roots of unity, proved a function field result about these enriched class groups and gave conjectural explicit formulas, expressed in terms of these pairings, for the case when $\Gamma=\Z/2 \Z$ and $K_0$ is an arbitrary number field.  Liu~\cite[Conjecture 1.2]{Liu2022} gave a moment conjecture which overlaps (and agrees) with Conjecture~\ref{C} in the case that $p$ does not divide the order of any unramified extension of $K_0$ (without deriving formulas for the probability distribution).
We explain the relation of our conjectures to those in \cite{Malle2008,Malle2010,Adam2015,Lipnowski2020} in Section \ref{S:MalleCompare}. In summary, our conjectures agree with prior conjectures in many, but not all, cases where they overlap.

Previous work that has been based primarily on empirical computations has faced the challenge that there are a limited number of situations that can be experimentally explored,  and as we see in this paper there are many parameters in these distributions, creating a challenge to generalization from limited data.  Further, some of the formulas for the distributions are quite involved, which makes it challenging to guess a formula from data.  Our results in the function field case have the benefit of working for an arbitrary $\Gamma$ 
and number of roots of unity in the base rational function field, which allows us to see many more cases at once from this point of view than can be explored with empirical data.  Further, our new methods of producing distributions explicitly from their moments  produce formulas even when those formulas are complicated.
Still, we expect empirical data in the number field case to be an important next step in further understanding of the situation.  
In Section~\ref{S:data}, we explain how all of Malle's previous empirical data on class group distributions supports Conjecture~\ref{C}, including data from a case in which Malle did not make any conjecture.  We also discuss further cases in which it would be interesting to have empirical data.

\subsection{Further remarks}

\begin{remark}
Conjecture~\ref{C} is not precise, in that it does not specify an ordering on $\mathcal{E}$ so that the distribution of class groups is well-defined.
Cohen and Martinet \cite{Cohen1990} order fields by $\Nm\Disc K$, but this is known to not work in general, even when there are not relevant roots of unity in the base field \cite[p. 929]{Bartel2020}.   One possible ordering, as suggested by Bartel and Lenstra in \cite{Bartel2020} and also by Theorem~\ref{T:MainFF},   is by $\Nm\sqrt{\Disc (K/K_0)}$, where $\sqrt{}$ denotes the radical.
We certainly imagine the conjecture only holding for orderings such that the proportion of fields in $\mathcal{E}$ containing any fixed field $K_1\not\sub K_0$ is $0$.
\end{remark}

\begin{remark}\label{R:scatinf}
Note that these extensions in $\mathcal{E}$ in Conjecture \ref{C} are split completely at all infinite places of $K_0$, because either $p^r>2$ and all infinite places of $K_0$ are complex, or $p=2$, which implies $|\Gamma|$ is odd. 
\end{remark}

\begin{remark}\label{R:rind}
If we consider modules with the distribution from Conjecture~\ref{C},  and then take their $p^k$-torsion, we have a distribution on $p^k$-torsion $S$-modules
that does not depend on $r$ as long as $r\geq k$.  
So, for example, if we are only asking about the $p$-torsion (equivalently the groups mod $p$) then the conjectured distribution does not depend on $r$ for $r\geq 1$.   This phenomenon was noticed empirically by Malle for the $2$-ranks of cyclic cubic extensions of $K_0=\Q$ versus $K_0=\Q(i)$
\cite[Section 6.4]{Malle2010}.  
We can see this phenomenon explicitly in the formulas of Theorem~\ref{T:formula}.
Even before working out formulas for the distribution, this follows from the fact that the moments indexed by $p^k$-torsion groups do not depend on $r$ for $r\geq k$, and since these moments determine a unique distribution of $p^k$-torsion modules by Theorem~\ref{T:mom}, this distribution on $p^k$-torsion modules will not depend on $r\geq k$. 
\end{remark}

From Theorem~\ref{T:ptors}, one can work out, for a $V$ from the distribution $\nu$, the distribution of $V$ as an $\F_p$-vector space.  It is  simple to write down the classical moments of $|V|$.  
For a finite representation  $V$ of $\Gamma$ over $\F_p$, we have $$|V|^k=|\Hom_{\F_p}(V,\F_p^k)|=|\Hom_{\Gamma}(V,\Hom(\F_p[\Gamma],\F_p^k))|=|\Hom_{\Gamma}(V,\F_p[\Gamma]^k))|$$
(by adjointness and the fact that permutation representations are self-dual).
Every homomorphism of $\Gamma$-representations $ V\ra \F_p[\Gamma]^k$ is a surjection onto exactly one subrepresentation of $\F_p[\Gamma]^k$.
Thus our conjecture implies that the average of $| \Cl_{K|K_0} [p]|^k $ is
$$
\sum_{\substack{V\sub \F_p[\Gamma]^k \\ \textrm{sub-$\Gamma$-rep.} \\V^{\Gamma}=1  }} 
\frac{|(\wedge^2 V)^\Gamma|}{|V|^{u}},
$$
a value straightforward to write down from the representation theory of $\Gamma$ over $\F_p$.

If one wants to make a conjecture on the product of the Sylow $p$-subgroups of $\Cl_{K|K_0}$ for a finite set of
primes $p$ not dividing $|\Gamma|$, then Theorem~\ref{T:MainFF} suggests the distribution with moments
$$
\frac{|(\wedge^2_{\Z}V)^\Gamma[|\mu(K_0)|]|}{|V|^u},
$$
where $\mu(K_0)$ is the group of roots of unity of $K_0$.  (These moments are multiplicative over Sylow $p$-subgroups.)
The same argument as in the proof of Theorem~\ref{T:formula} shows there is a unique such measure, and the formulas are as in Theorem~\ref{T:formula} with an additional product over $p$.
For infinitely many primes, we would still conjecture that the class group distributions 
weakly converge to the unique distribution with the moments above
(see \cite{Sawin2022} for the relevant topology and the uniqueness theorem), but because of the definition of weak convergence,
that is equivalent to making the conjecture for every finite set of primes. 

As explained in \cite{Wang2021}, the distributions of class groups of Galois extensions, for $p\nmid |\Gamma|$, can
be used to get distributions of class groups of non-Galois extensions.  We explain the implication of Conjecture~\ref{C}
for distributions of class groups of non-Galois extensions in Section~\ref{S:NG}.

There are recent developments \cite{Smith2022,Smith2022a} and interesting questions about distributions of the Sylow $p$-subgroups of $\Cl_K$ for $p\mid|\Gamma|$, including about the interactions with roots of unity in the base field, but we will not address those questions here. 
 
In a forthcoming paper, the authors will extend our function field results and number field conjectures to the non-abelian analog of the class group, the Galois group of the maximal unramified extension.  One feature of this forthcoming treatment is that we will consider additional structure on the class group (and its non-abelian analog), as in \cite{Sawin2022a},  such that the formulas for the distribution with this additional structure will be nicer in certain ways. 

\subsection{Outline of the paper}
In Section~\ref{S:example}, we give the example of $\Gamma=C_7$, for which we give the formulas for the distribution in full and one can see the non-independence of dual representations.  In Section~\ref{S:FF}, we prove the function field $q\ra\infty$ moments that motivate Conjecture~\ref{C}.
In Section~\ref{S:Heuristic}, we explain the heuristic assumptions that lead from our function field theorem to Conjecture~\ref{C}.
In Section~\ref{S:mom}, we review the results on the moment problem for modules that we will need to apply to find formulas for our conjectured distribution from our conjectured moments.  In Sections~\ref{S:vs} and \ref{S:DVR}, we make some calculations of distributions on vector spaces and modules over DVR's that will be necessary for our final formulas.  In Section~\ref{S:Gamma}, we  obtain formulas for the distribution of Conjecture~\ref{C}.  In Section~\ref{S:dontappear}, we show that the modules that we conjecture appear as class groups with density $0$ actually never appear as class groups. In Section~\ref{S:NG}, we show how our conjectures for class group of Galois extensions lead to explicit conjectures for the distribution of class groups of non-Galois extensions.  In Section~\ref{S:data}, we discuss the match between Malle's number field data and our conjectures (which is excellent).  In Section~\ref{S:MalleCompare}, we compare our conjectures to conjectures made by Malle, and Adam and Malle.

\subsection{Acknowledgements}

We would like to thank Yuan Liu and John Voight for helpful conversations.  We would like to thank J\"{u}rgen Kl\"{u}ners, Yuan Liu, Gunter Malle, and an anonymous referee for helpful comments on an earlier version of this manuscript.
The second author would like to thank Ben Breen, Ila
Varma, and John Voight for interesting conversations on this question
during their joint investigation.
Will Sawin was supported by NSF grant DMS-2101491 while working on this paper.
Melanie Matchett Wood was  partially supported by a Packard Fellowship for Science and Engineering,  NSF CAREER grant DMS-1652116, and NSF Waterman Award DMS-2140043 while working on the paper.
She was also a Radcliffe Fellow during part of this work, and thanks the Radcliffe Institute for Advanced Study for their support. 
\subsection{Notation}
We write $\mathbb{N}$ for the non-negative integers.

We write $\F_q$ for the finite field with $q$ elements.

For a group $G$ with an action of $\Gamma$, we write $G^{\Gamma}$ for the invariants.  For a afinite abelian group $A$, we write $A[p^\infty]$ for the Sylow $p$-subgroup of $A$.

For an extension $K/K_0$, we let  $\Cl_{K|K_0}$ be the quotient of the class group $\Cl_{K}$ of $K$ by the image of $\Cl_{K_0}$ under inclusion of ideals from $K_0$ to $K$.
(Note for a prime $p\nmid [K:K_0]$, we have that $\Cl_{K_0}[p^\infty]\ra \Cl_{K}[p^\infty]$ is an injection, since taking the norm map and dividing by $[K:K_0]$ provides an inverse.)

A \emph{$\Gamma$-extension} $K/K_0$ is a Galois field extension $K/K_0$, along with a choice of isomorphism $\iota: \Gal(K/K_0)\isom \Gamma$.
An isomorphism of $\Gamma$-extensions $(K,\iota)$, $(K',\iota')$ is given by a field isomorphism $\phi: K\ra K'$, fixing each element of $K_0$ and such that the induced map
$\phi_*: \Gal(K/K_0)\ra \Gal(K'/K_0)$ satisfies $\iota=\iota'\circ \phi_*$.
When $K/K_0$ is a $\Gamma$-extension,  using $\iota$,
there is an associated action of $\Gamma$ on $\Cl_{K|K_0}$.  
Further, 
$\Cl_{K|K_0}$ is a $\Z[\Gamma]/(\sum_{\gamma\in \Gamma}\gamma)$-module, 
since $(\sum_{\gamma\in \Gamma}\gamma)I=\Nm_{K/K_0} I$.

When $K$ is a function field, and $\O_K$ is a maximal order in $K$, we write $\Cl(\O_K)$ for the quotient of the ideal group of
$\O_K$ by the principal ideals of $\O_K$.

Given a commutative ring $R$ and an $R$-module $M$, we write $\wedge^2_R M$ for the $R$-module
quotient of $M\tensor_R M$ by elements of the form $m\tensor m$.  We write $\wedge^2 M$ for $\wedge^2_{\Z} M$

We write $\Sur(A,B)$ to denote the set of surjections from $A$ to $B$ in the appropriate category.  

We also use $\Aut(M)$ to denote automorphisms in the appropriate category, e.g. of $R$-modules if $M$ is an $R$-module.
We will write $\Aut_R(M)$ if we think there is a possibility of confusion.

For a finite set $S$, we write $|S|$ for the number of elements of $S$.

We write $C_n$ for the cyclic group of order $n$.

A product $\prod_{i=n}^{n-1} a_n$ is always $1$ by convention.

We write $\eta_q(k):=\prod_{i=1}^k(1-q^{-i})$ and also write $\eta(k)=\eta_q(k)$ when $q$ is clear from context, and $\eta(0)=1$.

\section{Example: $\Gamma=C_7$}\label{S:example}

In this section, we consider the example $K_0=\Q$ and $\Gamma=C_7$ and $p=2$.    
We have that $S=\Z_2[\Gamma]/(\sum_{\gamma\in\Gamma} \gamma)=R_1 \times R_2$, where
each $R_i$ is isomorphic to the ring of integers in the degree $3$ unramified extension of $\Q_2$.  
There are three irreducible representations of $C_7$ over $\F_2$ (or $\Q_2$), the trivial representation, and $V_1,V_2$, two three dimensional dual representations
(that correspond to $R_1, R_2$ respectively).

For $K/\Q$ a $C_7$-extension,  $\Cl_K [2^\infty]$ is an $S$-module.  For an $S$-module $V$, we let
$a_1=a_1(V)$ denote the $R_1/2R_1$-rank of $V/(2,R_2)V$, i.e. the multiplicity of $V_1$ in $V/2V$.  We define $a_2=a_2(V)$ similarly.  
Conjecture~\ref{C} predicts that among such $K$, the $S$-module $\Cl_K [2^\infty]$ is distributed according to $\nu$, where
\[
\nu(\{ V \})= \frac{8^{a_1a_2}}{|V||\Aut(V)|}\prod_{i=1}^\infty(1-8^{-i})  \prod_{i= 1}^{ a_1 } (1 - 8^{ - 1- i } )
\prod_{i=  1}^{ a_2} (1 - 8^{ - 1- i} ) \cdot  \begin{cases} 1 + 8^{-1} &  \textrm{if  }
a_1=a_2 \\  1 & \textrm{if } |a_1 - a_2 | = 1 \\ 0 & \textrm{if } |a_1 -a_2|>1. \end{cases} \]
(Corollary~\ref{C:specific} gives this formula.)
For comparison to the moments, note $|(\wedge^2_{\Z_p}V)^\Gamma[p]|=8^{a_1a_2}$.
One can give an explicit formula for $V$ and $|\Aut(V)|$ (see proof of Proposition~\ref{P:DVRselfdual}
or \cite[II (1.6)]{Macdonald2015}).

From the above formulas, one can also work out the distribution of $V$ as a $\Z_2$-module.  As an example, we work out the distribution of $V/2V$, so the predicted distribution of $\Cl_K [2]$, from Theorem~\ref{T:ptors}.  Since $R_i/2R_i\isom \F_8$,  for a finite $S$-module $V$,  we have that $V/2$ is a vector space over $\F_2$ with dimension $3k$ for some $k$.
For $V$ in the support of $\nu$, we see that if $k$ is even,  we must have $a_1=a_2=k/2$, and if $k$ is odd, then either 
$a_1=(k-1)/2$ and $a_2=(k+1)/2$,  or $a_1=(k+1)/2$ and $a_2=(k-1)/2$.
Thus we have
\begin{align*}
&\nu(\{ V \,|\,  |V/2V|=8^k  \})=\\
&\begin{cases}
8^{-k^2/4-k}  (1-8^{-2}) \prod_{i=1}^{k/2} (1-8^{-i})^{-2}
{\prod_{i=2}^\infty(1-8^{-i})}
& 
\textrm{if $k$ even}
\\
2\cdot 8^{-(k^2+3)/4-k} {\prod_{i=1}^{(k+1)/2}(1-8^{-i})^{-1}} {\prod_{i=1}^{(k-1)/2}(1-8^{-i})^{-1}} {\prod_{i=1}^\infty(1-8^{-i})}
 & \textrm{if $k$ odd. }
 \end{cases}
\end{align*}
(To calculate $| \Aut(V/2V)|$, we observe that $V/2V \cong \mathbb F_8^{a_1} \times \mathbb F_8^{a_2}$ viewed as module over $\mathbb F_8 \times \mathbb F_8$, so its automorphism group is $\Aut_{\mathbb F_8}(\mathbb F_8^{a_1}) \times \Aut_{\mathbb F_8}(\mathbb F_8^{a_2}) $, and then apply the formula $|\Aut_{\F_q}(\F_q^n)|=q^{n^2}\prod_{i=1}^n(1-q^{-i})$.)

\section{Function field results}\label{S:FF}

By analyzing components of Hurwitz spaces, Liu,  Zureick-Brown, and the second author \cite{Liu2019} have found 
$q\ra\infty$ moments of Galois groups of maximal unramified extensions of $\Gamma$-extensions of $\F_q(t)$.  In this paper, we focus on
the class group distributions (i.e. the abelianizations of those Galois groups),  and we give such results in this section.

For a finite group $\Gamma$, let $E_{\Gamma}(n,q)$ be the set of isomorphism classes of extensions $K/\F_q(t)$ with a choice of isomorphism $\Gal(K/\F_q(t))\isom \Gamma$, that are split completely above $\infty$ and such that the radical of the discriminant ideal $\Disc(K/\F_q(t))$ has norm $q^n$.
For such a $K$, let $\O_K$ be the maximal order of $K$ over $\F_q[t]$ and let $\Cl(\O_K)$ denote its class group.

\begin{theorem}\label{T:MainFF}
Let $\Gamma$ be a finite group and $H$ be a finite abelian group with action of $\Gamma$
such that $(|H|,|\Gamma|)=1$ and $H^\Gamma=1$.  Let $h$ be an integer such that $h| |H|$.
Then
$$
\lim_{x\ra\infty} \lim_{\substack{q\ra\infty\\(q,|\Gamma||H|)=1\\(q-1,|H|)=h}} 
\frac{\sum_{n\leq x} \sum_{K\in E_\Gamma(n,q)} |\Sur_\Gamma( \Cl(\O_K),H)|  }{\sum_{n\leq x} |E_\Gamma(n,q)|}
 =\frac{|(\wedge^2 H[h])^\Gamma| }{|H|}
,$$
where in the limit $q$ is always a prime power.
\end{theorem}

Theorem~\ref{T:MainFF} follows from \cite[Theorem 1.1]{Liu2022}, which is a more general and refined moment theorem, 
but we give a short argument here for just the statement of Theorem~\ref{T:MainFF}.
 See also Corollary~\ref{C:MainFF} for the implication of these averages on the distribution of $\Cl(\O_K)$.

\begin{proof}
This is \cite[Theorem 1.4]{Liu2019}, except that we allow $(q-1,|H|)$ to vary, and there is an additional factor of $|(\wedge^2 H[n])^\Gamma| $ on the right-hand side, along with some simplifications because $H$ is abelian.  A finite abelian group with an action of $\Gamma$ is admissible in the sense of \cite{Liu2019} if and only if it has order prime to $|\Gamma|$ and no $\Gamma$-invariants.  Also, via class field theory,  the group $\Gal(K^\#/K)$ considered in 
\cite{Liu2019} has abelianization $\Cl(\O_K)$.

The result \cite[Theorem 1.4]{Liu2019} follows in a straightforward way from \cite[Theorem 10.4]{Liu2019}, and our theorem here follows
in exactly the same way from an analog of \cite[Theorem 10.4]{Liu2019} in which the hypothesis that $(q-1,|H|)=1$ is replaced by $(q-1,|H|)=h$
and the component counting result 
$
 \pi_{G,c}(q,n) =\pi_{\Gamma}(q,n)+O_G( n^{d_{\Gamma}(q)-2});
$
is replaced by
$$
 \pi_{G,c}(q,n) =|H_2(H,\Z)^{\Gamma}[h]| \pi_{\Gamma}(q,n)+O_G( n^{d_{\Gamma}(q)-2}).
$$
(Recall when $H$ is abelian, $H_2(H,\Z)\isom \wedge^2 H$.)

This analog of \cite[Theorem 10.4]{Liu2019} follows almost entirely the proof of  \cite[Theorem 10.4]{Liu2019}, with just a change at the very end.
None of the intermediate results require that $(q-1,|H|)=1$.  However,  at the very end of the proof of \cite[Theorem 10.4]{Liu2019}, it is said that 
``Since the kernel of the map $f$  
[$: \ker(\overline{S}^1\ra G_1)\ra \ker(\overline{S}^2 \ra G_2)$] 
on the right has order relatively prime to $q-1$, any element of $\ker(\overline{S}^2\ra G_2)$ has the same number of $(q^{-1}-1)$th roots as 
any preimage in $\ker(\overline{S}^1\ra G_1)$.''  
(We interpret $(q^{-1}-1)$ modulo the order of group in which we are taking roots.)
In the more general setting when $q-1$ and $|H|$
are not relatively prime, we need to account for the difference in the number of $(q^{-1}-1)$th roots.

For any finite group $|H|$ (not necessarily abelian) with $(|H|,|\Gamma|)=1$, we have an exact sequence
$$
1\ra H_2(H,\Z)^\Gamma \ra H_2(H\rtimes \Gamma,\Z) \ra H_2(\Gamma,\Z) \ra 1
$$
(e.g.  from the Lyndon-Hochschild-Serre spectral sequence and the fact that $(|H|,|\Gamma|)=1$ implies invariants and co-invariants are the same).
Since $H_2(H,\Z)$ and $H_2(\Gamma,\Z)$ are of relatively prime order, this exact sequence  splits canonically.
The group $\ker(\overline{S}^1\ra G_1)$ mentioned above is a quotient of $H_2(H\rtimes \Gamma,\Z)$ by a certain subgroup (referred to in 
\cite[Section 12]{Liu2019} as $Q_{c_1}$) generated by elements that are in the image of a map from $H_2((\Z/|\Gamma|\Z)^2,\Z)$ 
(since each $x\in c_1$ has order dividing $|\Gamma|$).  The group $\ker(\overline{S}^2\ra G_2)$ mentioned above is a quotient of $H_2(\Gamma,\Z)$
by the image of this certain subgroup ($Q_{c_2}$, which is the image of $Q_{c_1}$).   So the kernel of the map $f$ is  $H_2(H,\Z)^\Gamma$,
and in general any element of $\ker(\overline{S}^2\ra G_2)$ has $|H_2(H,\Z)^\Gamma[q-1]|$ times as many $(q^{-1}-1)$th roots as 
any preimage in $\ker(\overline{S}^1\ra G_1)$.  This is the only change needed to prove the analog of \cite[Theorem 10.4]{Liu2019}
described above, which can then be used to prove this result exactly as in the proof of \cite[Theorem 1.4]{Liu2019}.
\end{proof}

\section{Heuristic argument}\label{S:Heuristic}

In this section, we explain the heuristic argument that leads to Conjecture~\ref{C}.

The Cohen-Martinet conjectures \cite{Cohen1990} include some predictions about Sylow $p$-subgroups of relative class groups of $\Gamma$-extensions when $p\mid \Gamma$, but in this paper we only consider the case when $p\nmid |\Gamma|$.
In this case,  Cohen and Martinet make conjectures about the distribution of $\Cl_{K|K_0}[p^\infty]$ for a base field $K_0$
and the family of all $\Gamma$-extensions $K/K_0$ with specified behavior at the infinite places of $K_0$.  
In this paper,  we only consider families $\mathcal{E}$ of $\Gamma$-extensions $K/K_0$ in which all infinite places of $K_0$ are required to split completely.  (See Remark~\ref{R:scatinf}.)
In this case, the only feature of $K_0$ that Cohen and Martinet's conjectures \cite{Cohen1990} take as input is the number $u$ of infinite places of $K_0$ (see, e.g.,  \cite[Theorem 4.1]{Wang2021}).
As Malle's computations \cite{Malle2008,Malle2010} suggest and Theorem~\ref{T:MainFF} shows (in the function field case) one must also take into account the maximal $r$ such that  $K_0$ contains the $p^r$th roots of unity.
One heuristic assumption we make here are that $u$ and $r$ are the only relevant inputs from the base field into the conjectures. (We hope that conceptual evidence for this heuristic will eventually be found in the function field case. If moments are calculated over a function field of arbitrary genus, and shown to depend only on $u$ and $r$, this heuristic assumption would be more justified, albeit still depending on the reliability of analogy between number fields and function fields.) 

In the function field case, we take the class group of a ring $\O_K$ whose fraction field is $K$, but unlike the number field case, there is no canonical choice of such a ring.  Instead, we pick a set of places of $K$ that we consider ``infinite places'' for the sake of the analogy, and we let $\O_K$ be the elements of $K$ that have non-negative valuation at all non-infinite places of $K$.  
Via this analogy, we expect there to be uniform behavior across number fields and function fields (taking into account their number of infinite places, and the roots of unity that they contain).  

However, in the function field case,  we can change our notion of ``infinite places'' by adding an additional place $v$ of $K_0$ to the list of infinite places.  The new family $\mathcal{E}'$ of $\Gamma$-extensions will be those from the original family $\mathcal{E}$ that are split completely at $v$.  It is generally expected that such a restriction at a non-archimedean place does not change class group statistics (e.g. see  \cite[Theorem 1]{Bhargava2015d}, \cite[Corollary 4]{Bhargava2016}, \cite[Conjecture 1.4]{Wood2018}).   (We hope that future results will be proven in the function field case that will more fully justify this expectation.) Then,  $\Cl_{\O'_K|\O'_{K_0}}$ (for the new $\O'_K,\O'_{K_0}$) after the inclusion of $v$ as an infinite place,  is the quotient, as a $\Z[\Gamma]$-module, of the original $\Cl_{\O_K|\O_{K_0}}$ by any prime of $K$ over $v$.
We assume the heuristic that this prime is distributed uniformly in $\Cl_{\O_K|\O_{K_0}}$. (See \cite{Klagsbrun2017,Klagsbrun2017a,Wood2018} for some conjectures and results in this direction. We hope that future results will be proven in the function field case that will more fully justify this heuristic.)
If $X$ is a random $\Z[\Gamma]$-module with $\E(|\Sur_\Gamma(X,G)|)=M_G$ for all finite $\Z[\Gamma]$-modules $G$, and if we let $X'$ be the quotient of $X$ by a uniform random element of $X$, then
$\E(|\Sur_\Gamma(X',G)|)=M_G/|G|$.  Theorem~\ref{T:MainFF} gives (in a large $q$ limit) moments of $\Cl(\O_K)$ assuming only one infinite place, and so by the reasoning above, we expect if we modified so as to  define $\O_K$ using $u$ infinite places of $\F_q(t)$, then we would replace the $|H|$ in the denominator of Theorem~\ref{T:MainFF} with $|H|^u$.  (Again, we hope to see such a result proven in the future.  See also \cite[Sections 6.2 and 6.3]{Liu2022} for different theoretical heuristics for the $|H|^u$ factor.)

Heuristically further  we assume that the limits in Theorem~\ref{T:MainFF} (and the extension suggested above to $u$ infinite places) can be exchanged, as was proven in some special cases in \cite{Ellenberg2016}.  Since we expect every base field $K_0$ with a fixed number $u$ of infinite places and $p^r$th but not $p^{r+1}$th roots of unity to give the same class group distributions,  if we consider all $q$ such that $q-1$ has a fixed valuation at $p$, then we do not expect the limit in $q$ to change the average, and in particular we expect 
Theorem~\ref{T:MainFF} (and the extension suggested above to $u$ infinite places) to hold without the limit in $q$.
Additionally, we expect the same result for number fields with a given $u$ and $r$, and this  leads precisely to the statement of Conjecture~\ref{C}.

\section{Results on the moment problem for $R$-modules}\label{S:mom}
In \cite{Sawin2022}, we give a result to construct a distribution from its moments in a rather general category.  
In this section, we review that result focusing on the category of finite $R$-modules for a DVR $R$ with maximal ideal $\mathfrak{m}$, and very similar categories.

Let $R$ be either a DVR with finite residue field, a quotient ring of a DVR with finite residue field, or a finite product of such rings.
Let $\hat{R}$ denote the product of the completions of $R$ at its maximal ideals $\mathfrak{m}_1,\dots,\mathfrak{m}_r$.  
Let $\mathcal{P}$ be the set of isomorphism classes of finitely generated $\hat{R}$-modules.  
For $\underline{k}=(k_1,\dots,k_r)\in \mathbb{N}^r$, let $\mathfrak{m}^{\underline{k}}=\prod_i \mathfrak{m}_i^{k_i}$
and for $X$ an $\hat{R}$-module let $X^{\leq \underline{k}}=X/\mathfrak{m}^{\underline{k}}X$.
We consider a topology on $\mathcal{P}$ generated by $\{X \,|\,  X^{\leq \underline{k}} \isom N\}$ for $\underline{k}\in \mathbb{N}^r$ and $N$
a finite $ R/\mathfrak{m}^{\underline{k}}$-module  (and we take the associated Borel $\sigma$-algebra on $\mathcal{P}$).

The isomorphism classes of finite $R$-modules (equivalently, finite $\hat{R}$-modules) are in bijection with 
$r$-tuples $\underline{\lambda}=(\lambda^1,\dots,\lambda^r)$ of 
finite partitions $\lambda^i=(\lambda^i_1\geq \lambda^i_2\geq\dots)$, where the parts of $\lambda^i$ are bounded above by the
minimal positive integer $a_i$ such that $\mathfrak{m}_i^{a_i}=0$.
(Here, $\underline{\lambda}$ corresponds to the module $N_{\underline{\lambda}}:=\oplus_{i,j} R/\mathfrak{m_i}^{\lambda^i_j}$. )
 We write $\lambda'$ for the conjugate of a partition.  The operation $N\mapsto N^{\leq \underline{k}}$ simply truncates the $i$th associated conjugate partitions after the first $k_i$ terms.
We say a module $N$ associated to $\underline{\lambda}$   is semi-simple if $(\lambda^i)'_2=0$ for all $i$, i.e.
$N\isom \prod_i (R/\mathfrak{m_i})^{e_i}$ for some positive integers $e_i$.

Let $q_i=|R/{\mathfrak{m}_i}|$.
For two $r$-tuples of partitions $\underline{\lambda},\underline{\pi}$ and a surjection $g: N_{\underline{\pi}} \ra N_{\underline{\lambda}}$
 we define $\mu(g)$ to be $0$ if $\ker g$ is not semi-simple and $\prod_i (-1)^{e_i}q_i^{\binom{e_i}{2}}$ if $\ker g\isom \prod_i (R/\mathfrak{m_i})^{e_i}$. 
 For finite $R$-modules $N,P$, we define $\hat{\mu}(N,P)=\sum_{g\in \SSur(P,N)/\Aut(N)} \mu(g)$.

For a finite $R$-module $N$, the \emph{$N$-moment}
 of a measure $\nu$ on $\mathcal{P}$ is defined to be
$$
\int_{X\in \mathcal{P}} |\SSur_R(X,N)| d\nu.
$$
Given values $M_N$ for each finite $R$-module $N$, for each  $\underline{k}\in\mathbb{N}^r$ and 
 finite $R/\mathfrak{m}^{\underline{k}}$-module $N$, we define (what will be the formulas for our measures)
\begin{equation}\label{E:defv}
v_{\underline{k},N}:=\sum_{\substack{ P \textrm{ finite}\\ \textrm{$R/\mathfrak{m}^{\underline{k}}$-module}}} \frac{\hat{\mu}(N,P)}{|\Aut(P)|} M_P,
\end{equation}
and we say the values $M_N$ are \emph{well-behaved} (for $R$) if the sum \eqref{E:defv} is absolutely convergent for each $\underline{k}\in\mathbb{N}^r$ and 
 finite $R/\mathfrak{m}^{\underline{k}}$-module $N$.

\begin{theorem}[{\cite[Theorems 1.7, 1.8]{Sawin2022} }]\label{T:mom}
Let $R$ be a finite product of quotients of DVRs with finite residue fields and, for each finite $R$-module $N$, let $M_N\in \R$ such that
the values $M_N$ are well-behaved.  Then we have the following.
\begin{enumerate}
\item[(Existence):] There is a measure on $\mathcal{P}$ with $N$-moment $M_N$ for each finite $R$-module $N$ if and only if the $v_{\underline{k},N}$ defined above are non-negative for each $\underline{k}$ and $N$.
\item[(Uniqueness):] When such a measure $\nu$ exists, it is unique and is given by the formulas
$$
\nu(\{X \,|\, X^{\leq \underline{k}} \isom N\}) =v_{\underline{k},N}.
$$
\item[(Robustness):] If $\nu^t$ are measures on $\mathcal{P}$ such that for each finite $R$-module $N$,
$$
\lim_{t\ra\infty} \int_{X \in \mathcal{P}} |\SSur(X,N)| d\nu^t= M_N, 
$$
then a measure $\nu$ with moments $M_N$ exists and
the  $\nu^t$ weakly converge to $\nu$.
\end{enumerate}
\end{theorem}

\begin{remark}
We give some notes on translation from the more general language of \cite{Sawin2022}.
In \cite[Lemma 6.1]{Sawin2022}, it is explained that the category of $R$-modules is a ``diamond category'' and the ``levels'' of \cite{Sawin2022}
correspond to modules over the finite quotient rings $R/\mathfrak{m}^{\underline{k}}$ of $R$.
The results \cite[Lemma 5.7, Lemma 5.10 (2)]{Sawin2022} show that 
$\mathcal{P}$ is as described here, and \cite[Lemmas 2.6, 3.8]{Sawin2022} show that $\mu$ is as defined here.
The definition of ``well-behaved'' in  
\cite{Sawin2022} involves an additional factor of $Z(\pi)^3$, which
by \cite[Lemma 3.11]{Sawin2022} is at most $8^r$ in our case and thus does not affect absolute convergence.

\end{remark}

The measures in this paper will turn out to be supported on finite $R$-modules, which we will show using the formulas provided by
Theorem~\ref{T:mom}.  Then we can use the following statement, weaker than Theorem~\ref{T:mom}, but simpler to think about
 because it only involves finite $R$-modules.

\begin{corollary}\label{C:finsup}
Let $R$ be a finite product of quotients of DVRs with finite residue fields and, for each finite $R$-module $N$, let $M_N\in \R$ such that
the values $M_N$ are well-behaved and are the moments of a measure $\nu$ supported on finite $R$-modules.  Then if $\nu^t$ are measures on 
the set of isomorphism classes of finite $R$-modules such that for each finite $R$-module $N$,
$$
\lim_{t\ra\infty} \int_{X \in \mathcal{P}} |\SSur(X,N)| d\nu^t= M_N, 
$$
then  the $\nu^t$ weakly converge to $\nu$ (for the discrete topology).
\end{corollary}

\section{Measures on the set of isomorphism classes of $\F_q$ vector spaces}\label{S:vs}
In this section, we compute some distributions from their moments that will be important for our eventual class group conjectures.

If  $q$ is a prime power, then $R=\F_q$ is a quotient of a DVR with finite residue field, so we can apply Theorem~\ref{T:mom}.  The category of finite $R$-modules is just the category of finite $\F_q$-vector spaces (and $\hat{R}=R$).

\begin{proposition}\label{P:selfdualvs}
Let $q$ be a prime power.
 Let $t>0$.  The values $M_{{\F_q}^\rho}=q^{\frac{\rho^2+\rho}{2}-t\rho}$ are well-behaved (for $R$) and are the moments of the measure $\nu$ (on isomorphism classes of finite $\F_q$-vector spaces) 
 such that 
$$
\nu({\F_q}^\lambda)=\frac{q^{-\binom{\lambda}{2}-t\lambda} } {\prod_{i=1}^\lambda(1-q^{-i})\prod_{i\geq 0}(1+q^{-i-t}) }=\frac{M_{{\F_q}^\lambda}}{|\Aut(\F_q^\lambda)|} 
\frac{1}{\prod_{i\geq 0}(1+q^{-i-t})}
$$
for all $\lambda\in \mathbb{N}$.
\end{proposition}

Recall $\eta(k):=\prod_{i=1}^k(1-q^{-i})$.

\begin{proof}
We apply Theorem~\ref{T:mom}.
Let $N=\F_q^\lambda$ and $P=\F_q^{\lambda+e}$.  
There are $q^{(\lambda+e)\lambda}$ homomorphisms $P\ra N$, a proportion $\eta(\lambda+e)/\eta(e)$ of these are epimorphisms, and for any of these epimorphisms $g$ we have
$\mu(g)=(-1)^eq^{\binom{e}{2}}$.
We have $|\Aut(\F_q^\rho)|= q^{\rho^2} \eta(\rho)$.   So,
$$
\frac{\hat{\mu}(\F_q^\lambda,\F_q^{\lambda+e})}{|\Aut(\F_q^{\lambda+e})|}= \frac{ (-1)^eq^{\binom{e}{2}} q^{(\lambda+e)\lambda} \eta(\lambda+e)
  }{ \eta(e)q^{\lambda^2} \eta(\lambda) q^{(\lambda+e)^2} \eta(\lambda+e)} = \frac{ (-1)^eq^{-\frac{e^2}{2}-\frac{e}{2}-e\lambda-\lambda^2} 
  }{ \eta(e)\eta(\lambda)  } .
$$
Thus
\begin{align*}
v_{1,N}:= &\sum_{e\geq 0} \frac{\hat{\mu}(\F_q^\lambda,\F_q^{\lambda+e})}{|\Aut(\F_q^{\lambda+e})|} q^{\frac{(\lambda+e)^2+(\lambda+e)}{2}-t(\lambda+e)}
    = \frac{q^{-\binom{\lambda}{2}-t\lambda}}{\eta(\lambda)}
  \sum_{e\geq 0} \frac{ (-1)^e} 
  { \eta(e)}  q^{-te}.
\end{align*}
We see that the sum is absolutely convergent since $t>0$,
 and thus the moments are well-behaved as claimed.
By the $q$-binomial theorem for negative powers,
$$
\sum_{e\geq 0} \frac{v^e}{\eta(e)}=\prod_{i\geq 0}(1-vq^{-i})^{-1},
$$
and so
$$
v_{1,N}=\frac{q^{-\binom{\lambda}{2}-t\lambda}}{\eta(\lambda)} \prod_{i\geq 0}(1+q^{-i-t})^{-1}.
$$
\end{proof}

Now consider $R=\F_q\times \F_q$,  and we write $R$-modules as pairs of $\F_q$-vector spaces.
 If one has moments $M_{(N_1,N_2)}$ that factor as 
$M_{(N_1,N_2)}=M'_{N_1} M''_{N_2}$, then the sums defining well-behavedness and the $v_{\underline{k},N}$ all factor into two factors, one corresponding to each factor of $R$, reducing the moment problem for the moments $M_{(N_1,N_2)}$ to understanding the moment problems for $M'_N$ and $M''_N$.
However, not all such moments necessarily factor.

\begin{proposition}\label{P:dualpairsvs}
Let $q$ be a prime power and $R=\F_q\times \F_q$.
Let $s_1$ and $s_2$ be integers with $s_1+s_2\geq 0$.
 Let $M_{({\F_q}^{\rho},{\F_q}^{\pi})}=q^{\rho \pi-s_1\rho-s_2\pi}$ for all $\rho, \pi\in \mathbb{N}$.
If $-s_1\leq \lambda-\phi \leq s_2$,
let 
\begin{align*}
D_q(\lambda,\phi,s_1, s_2)&:=\frac{q^{-\lambda^2-\phi^2+\lambda\phi-s_1\lambda-s_2\phi}\eta(\infty)\eta(s_1+s_2)}{\eta(\lambda)\eta(\phi)\eta(s_2-\lambda+\phi)\eta(s_1+\lambda-\phi)}\\ &=
\frac{M_{({\F_q}^{\lambda},{\F_q}^{\phi})}}{|\Aut({\F_q}^{\lambda},{\F_q}^{\phi})|}
\frac{\eta(\infty)\eta(s_1+s_2)}{\eta(s_2-\lambda+\phi)\eta(s_1+\lambda-\phi)} 
\end{align*} 
and let  $D_q(\lambda,\phi,s_1, s_2)=0$ otherwise.
Then the values $M_N$ are well-behaved for $R$ and  are the moments of the measure $\nu$ (on isomorphism classes of pairs of $\F_q$-vector spaces) such that for all non-negative integers $\lambda,\phi,$
$$
\nu(({\F_q}^{\lambda}, {\F_q}^{\phi}))=
D_q(\lambda,\phi,s_1, s_2).
$$
\end{proposition}

If $s_1=s_2=0$, then the distribution is supported on the pairs such that $\lambda=\phi$, and on those reduces to the distribution on $\F_q$-vector spaces
whose moments are all $1$ (see \cite[Lemma 6.3]{Sawin2022}, a distribution which, at least in the $\F_p$-case, would usually be called the Cohen-Lenstra distribution on $\F_p$ vector spaces, since it is the distribution conjectured by Cohen and Lenstra \cite[Sec. 9 (C5)]{Cohen1984} to be the distribution of $p$-torsion of class groups of imaginary quadratic fields (for odd primes $p$).

\begin{proof}
With reasoning as in Proposition~\ref{P:selfdualvs}, we have
\begin{align*}
v_{\underline{1},({\F_q}^{\lambda}, {\F_q}^{\phi})}:= 
 &\sum_{\trye,\tryf\geq 0 } 
\frac{ (-1)^{\trye+\tryf}q^{-\frac{\trye^2}{2}-\frac{\trye}{2}-\trye\lambda-\lambda^2 -\frac{\tryf^2}{2}-\frac{\tryf}{2}-\tryf\phi-\phi^2 } 
  }{ \eta(\trye)\eta(\tryf)\eta(\lambda)\eta(\phi)  } q^{(\lambda+\trye)(\phi+\tryf)-s_1(\lambda+\trye)-s_2(\phi+\tryf)}
\\
 =&
\frac{q^{-\lambda^2-\phi^2+\lambda\phi-s_1\lambda-s_2\phi}}{\eta(\lambda)\eta(\phi)} 
 \sum_{\trye,\tryf\geq 0 } 
\frac{ (-1)^{\trye+\tryf}q^{-\frac{\trye^2}{2}-\frac{\trye}{2}-\lambda \trye  +\phi \trye -s_1\trye-\frac{\tryf^2}{2}-\frac{\tryf}{2}-\phi \tryf +\lambda \tryf-s_2\tryf+\trye\tryf }
  }{ \eta(\trye)\eta(\tryf)  }.
\end{align*}
We can check that this sum converges absolutely 
as follows.
If we fix $\trye$, we have an inner sum $\sum_{\tryf} q^{h(\tryf)/2}$, where $h$ is a quadratic polynomial with integer coefficients
that takes its maximum values (on integers) at $\tryf=\trye-\phi+\lambda+s_2, \trye-\phi+\lambda+s_2-1$.  Thus the inner sum is bounded by a constant times the summand for $\tryf=\trye-\phi+\lambda+s_2$, which then gives an upper bound for the entire (absolute) sum by a geometric series that converges if $s_1+s_2+1>0$.

Now, we evaluate the sum over $\tryf$ in the equation for $v_{\underline{1},({\F_q}^{\lambda}, {\F_q}^{\phi})}$ above using  the $q$-binomial theorem for negative powers and obtain
\begin{align*}
v_{\underline{1},({\F_q}^{\lambda}, {\F_q}^{\phi})}
 =&
\frac{q^{-\lambda^2-\phi^2+\lambda\phi-s_1\lambda-s_2\phi}}{\eta(\lambda)\eta(\phi)} 
 \sum_{\trye\geq 0 } 
\frac{ (-1)^{\trye}q^{-\frac{\trye^2}{2}-\frac{\trye}{2}-\lambda \trye  +\phi \trye -s_1\trye}}{ \eta(\trye)}
\prod_{k\geq 0}(1-q^{-\phi+\lambda+\trye-1-s_2-k}).
\end{align*}
This final product is $0$ if $-\phi+\lambda+\trye-1-s_2\geq 0$, and 
otherwise is $\eta(\infty)/\eta(s_2-\lambda+\phi-\trye)$.
If $\lambda-\phi>s_2$, then we have  $v_{\underline{1},({\F_q}^{\lambda}, {\F_q}^{\phi})}=0$ and otherwise
\begin{align*}
v_{\underline{1},({\F_q}^{\lambda}, {\F_q}^{\phi})}
 =&
\frac{q^{-\lambda^2-\phi^2+\lambda\phi-s_1\lambda-s_2\phi}\eta(\infty)}{\eta(\lambda)\eta(\phi)} 
 \sum_{\trye= 0 }^{s_2-\lambda+\phi} 
\frac{ (-1)^{\trye}q^{-\frac{\trye^2}{2}-\frac{\trye}{2}-\lambda \trye  +\phi \trye -s_1\trye}}{ \eta(\trye)\eta(s_2-\lambda+\phi-\trye)}.
\end{align*}
This final sum can be evaluated by the $q$-binomial theorem (for $n\geq 0$)
$$
\sum_{k=0}^{n} q^{-\binom{k}{2}} \frac{\eta(n)}{\eta(n-k)\eta(k)} t^k =\prod_{k=0}^{n-1} (1+q^{-k} t),
$$
and we obtain
\begin{align*}
v_{\underline{1},({\F_q}^{\lambda}, {\F_q}^{\phi})}
 =&
\frac{q^{-\lambda^2-\phi^2+\lambda\phi-s_1\lambda-s_2\phi}\eta(\infty)}{\eta(\lambda)\eta(\phi)\eta(s_2-\lambda+\phi)} 
 \prod_{k=0}^{s_2-\lambda+\phi-1}(1-q^{-k-1-\lambda+\phi-s_1}).
\end{align*}
If  $\lambda-\phi <-s_1$, then we can check that one of the factors in the product above is zero, and otherwise we have
$$
 \prod_{k=0}^{s_2-\lambda+\phi-1}(1-q^{-k-1-\lambda+\phi-s_1})=\eta(s_1+s_2)/\eta(s_1+\lambda-\phi),
$$
which gives the proposition.
\end{proof}

\section{Distributions on isomorphism classes of modules over an unramified DVR}\label{S:DVR}
In this section, we compute some more distributions from their moments that will be important for our eventual class group conjectures.

Let $q$ be a power of a prime $p$, and throughout this section let $R$ be the ring of integers in the unramified extension of $\Q_p$ with residue field $\F_q$. Let $\mathfrak{m}$ be the maximal ideal of $R$.
For a partition $\lambda$, let $w(\lambda)=\sum_i (i-1)\lambda_i=\sum_j \binom{\lambda'_j}{2}=|\wedge^2_R N_\lambda|$.  We write $(1^e)$ to denote the partition with $e$ $1$'s.  We define $|\lambda|:=\sum_i \lambda_i$.
We write $N^{\leq k}:=N/\mathfrak{m}^k$

\begin{proposition}\label{P:DVRselfdual}
Let $q,R$ be as above.  Let $s,\epsilon$ be real numbers such that $s\geq 0$ and $s-\epsilon\geq 0$, and $r$ a positive integer, and let $M_{N_{\rho'}}=|\wedge_R^2 N_{\rho'}[p^r]| |N_{\rho'}[p^r]|^\epsilon
|N_{\rho'}|^{-s} =q^{\sum_{j=1}^{r} \left( \binom{\rho_j}{2} +\epsilon \rho_j \right)-s|\rho|}$ for all partitions $\rho$.
We define
$$
\Pf_q(t,m):=\sum_{e=0}^{m}
\frac{ 
(-1)^{ e} 
q^{-(t+1)e }\eta(m)
}
{ \eta(e)\eta(m-e)}.
$$
Then $\Pf_q(t,m)$ is positive for all real $t \geq 0$ and $m\in \mathbb N$.  Further, the values
$M_N$ are well-behaved for $R$ and  are the moments of the measure $\nu$ such that 
\begin{align*}
 \nu(\{P \,|\, P^{\leq k} \isom N_{\lambda'}\} 
) =\frac{q^{\sum_j \left(-\binom{\lambda_j}{2}-(s-\epsilon+1)\lambda_j\right) 
+\sum_{j\geq r+1} \left(-\binom{\lambda_j}{2}-\epsilon \lambda_j\right)}
}{\prod_{j\geq 2 }\eta(\lambda_{j-1}-\lambda_j)
} 
\left(\prod_{i\geq 0}(1+q^{-s+\epsilon-1-i})^{-1}  \right) & \\
 \times
\prod_{j= 2}^{\min(k,r) } 
\Pf_q(s-\epsilon,\lambda_{j-1}-\lambda_j)
\prod_{i=  \lambda_{k}+1}^{  \lambda_{\min(k,r)} } (1 - q^{ - s- i } )& \\
=\frac{M_{N_{\lambda'}}}
{|\Aut(N_{\lambda'})|} 
\prod_{i\geq 0}(1+ q^{-s+\epsilon-1-i})^{-1} 
\prod_{j= 2}^{\min(k,r) } 
\Pf_q(s-\epsilon,\lambda_{j-1}-\lambda_j)
\prod_{i=  \lambda_{k}+1}^{  \lambda_{\min(k,r)} } (1 - q^{ - s- i } )&
\end{align*}
for all $k\geq 1$ and $\lambda$ with $\lambda_{k+1}=0$.  This measure is supported on finite $R$-modules.  Finally, $\nu(\{N_{\lambda'}\})$ is given by the above expression for any $k>\lambda'_1$.
\end{proposition}

By convention, the product from $\lambda_k+1$ to $\lambda_{\min(k,r)}$  is $1$ if $\lambda_k =\lambda_{\min(k,r)}$.  
The result also is proven below for the case $r=\infty$, with the natural interpretations.

\begin{proof} 
If we let $a_e=\frac{ 
q^{-(t+1)e }
}
{ \eta(e)\eta(m-e)}$, we note that 
$$\frac{a_{e+1}}{a_e}=\frac{q^{-t-1}(1-q^{-m+e})}{(1-q^{-e-1})} <q^{-1} \frac{1}{1-q^{-1}}
\leq 1,
$$
since $t\geq0$ and $e\geq 0$ and $q\geq 2$.  So $\Pf_q(t,m)$ is an alternating series whose terms are decreasing in absolute value and whose first term is positive, which implies $\Pf_q(t,m)> 0.$

We now apply Theorem~\ref{T:mom}.
We let $N=N_{\lambda'}$.   Given a partition $\rho$, the number of surjections
$N_{\rho'} \ra N_{\lambda'}$ with simple kernel of type $(1^e)$, up to automorphisms of $N_{\lambda'}$, is 
$$
q^{\sum_j \binom{\rho_j}{2} -\sum_j \binom{\lambda_j}{2}-\binom{e}{2}} \prod_{j\geq 1} \frac{\eta(\rho_j-\rho_{j+1})}{\eta(\rho_j-\lambda_j)\eta(\lambda_j-\rho_{j+1})}
$$
if $|\rho|=|\lambda|+e$ and $\rho_j-\lambda_j,\lambda_j-\rho_{j+1}\geq 0$ for all $j\geq 1$; 
 and $0$ otherwise \cite[Ch. II (4.6)]{Macdonald2015}.  Each of these surjections has $\mu(N_{\lambda'},N_{\rho'})=(-1)^eq^{\binom{e}{2}}$.  
Let $\rho'$ contain $m_j$ $j$'s (so $m_j=\rho_j-\rho_{j+1}$ and $\rho_j=m_j+m_{j+1}+\dots$).
We have 
$|\Aut(N_{\rho'})|=q^{\sum_j \rho_j^2} \prod_j{\eta(m_j)}$ (\cite[II (1.6)]{Macdonald2015}). 
We have $|\wedge_R^2 N_{\rho'}[p^r]|=q^{\sum_{j=1}^{r} \binom{\rho_j}{2}}$.
Letting $\rho_j=\lambda_j+e_j$, we have
\begin{align*}
v_{k, N_{\lambda'}}&=\sum_{\substack{\trye_1\dots,e_k\geq 0\\
\forall i\geq 1: e_{i+1}\leq \lambda_i-\lambda_{i+1}
}} \frac{q^{\sum_j \binom{\rho_j}{2} -\sum_j \binom{\lambda_j}{2}-\binom{e}{2}}\prod_{j\geq 1} \frac{\eta(\rho_j-\rho_{j+1})}{\eta(\rho_j-\lambda_j)\eta(\lambda_j-\rho_{j+1})} (-1)^{e_j}q^{\binom{e_j}{2}} q^{\sum_{j=1}^{r} \left( \binom{\rho_j}{2} +\epsilon \rho_j \right)-s|\rho|} }
{q^{\sum_j (\rho_j)^2} \prod_{j\geq 1}{\eta(m_j)}}
\\
&=\sum_{\substack{\trye_1\dots,e_k\geq 0\\
\forall i\geq 1:e_{i+1}\leq \lambda_i-\lambda_{i+1}
}}  
\frac{ (-1)^{\sum_j e_j} q^{\sum_j \left(2\binom{\lambda_j+e_j}{2}-\binom{\lambda_j}{2}-(s-\epsilon)(\lambda_j+e_j)\right) 
-\sum_{j> r} \left(\binom{\lambda_j+e_j}{2} +\epsilon(\lambda_j+e_j) \right)
}
}
{q^{\sum_j (\lambda_j+e_j)^2} \prod_{j\geq 1}  \eta(e_j)\eta(\lambda_j-\lambda_{j+1}-e_{j+1})}
\\
&=\sum_{\substack{\trye_1\dots,e_k\geq 0\\
\forall  i \geq 1:e_{i+1}\leq \lambda_i-\lambda_{i+1}
}} 
\frac{ (-1)^{\sum_j e_j} q^{\sum_j \left(-\binom{\lambda_j}{2}-(s-\epsilon+1)\lambda_j-(s-\epsilon+1)e_j\right) -\sum_{j> r}  \left(\binom{\lambda_j+e_j}{2} +\epsilon(\lambda_j+e_j) \right)}
}
{\prod_{j\geq 1} \eta(e_j)\eta(\lambda_j-\lambda_{j+1}-e_{j+1})}\\
&=
\frac{q^{\sum_j \left(-\binom{\lambda_j}{2}-(s-\epsilon+1)\lambda_j\right) }}{\eta(\lambda_k)}
\left( \sum_{\trye_1\geq 0} (-1)^{\trye_1} \frac{q^{-(s-\epsilon+1)\trye_1}}{\eta(\trye_1)}  \right)
\prod_{j=2}^{\min(k,r)} 
\sum_{e_j=0}^{\lambda_{j-1}-\lambda_j}
\frac{ 
(-1)^{ e_j} 
q^{-(s-\epsilon+1)e_j }
}
{ \eta(e_j)\eta(\lambda_{j-1}-\lambda_{j}-e_{j})}\\
& \times 
 \prod_{j=
 \min(k,r)+1}^k \sum_{e_j=0}^{\lambda_{j-1}-\lambda_j}
\frac{ 
(-1)^{ e_j} 
q^{-(s-\epsilon+1)e_j -\binom{\lambda_j+e_j}{2} -\epsilon(\lambda_j+e_j)} 
}
{ \eta(e_j)\eta(\lambda_{j-1}-\lambda_{j}-e_{j})}
.
\end{align*}
The above sums are absolutely convergent when $s-\epsilon+1>0$, so we have well-behavedness. 
We use the $q$-binomial theorem for negative powers for the $e_1$ sum.
For $2\leq j \leq \min(k,r)$, the sum over $e_j$ is equal to $\Pf_q(s-\epsilon,\lambda_{j-1}-\lambda_j)/\eta(\lambda_{j-1}-\lambda_j)$.
Considering the sum over $e_j$ for $j> \min(k,r)$, we have 
\begin{align*}
\sum_{e_j=0}^{\lambda_{j-1}-\lambda_j}
\frac{ 
(-1)^{ e_j} 
q^{-(s-\epsilon+1)e_j -\binom{\lambda_j+e_j}{2}-\epsilon(\lambda_j+e_j)}
}
{ \eta(e_j)\eta(\lambda_{j-1}-\lambda_{j}-e_{j})} &=
q^{-\binom{\lambda_j}{2} -\epsilon\lambda_j}
\sum_{e_j=0}^{\lambda_{j-1}-\lambda_j}
\frac{ 
(-q^{-s-1-\lambda_j})^{ e_j} 
q^{-\binom{e_j}{2}} 
}
{ \eta(e_j)\eta(\lambda_{j-1}-\lambda_{j}-e_{j})} 
\\
&=
\frac{q^{-\binom{\lambda_j}{2}-\epsilon\lambda_j}}{\eta(\lambda_{j-1}-\lambda_j)}
\prod_{i=0}^{\lambda_{j-1}-\lambda_j-1}(1-q^{-s-1-\lambda_j  -i})
\end{align*}
by the $q$-binomial theorem. 
We can combine the products above over different $j$ and have
$$
\prod_{j= \min(k,r)+1}^k
\prod_{i=0}^{\lambda_{j-1}-\lambda_j-1}(1-q^{-s-1-\lambda_j  -i})=
\prod_{i=  \lambda_k+1}^{  \lambda_{ \min(k,r)} } (1 - q^{ - s- i } ).
$$
We put these identities all together, pull out all the powers of $q$, and collect the $\eta(\lambda_{j-1}-\lambda_j)$ terms in the denominator
to  obtain the equation for $\nu(\{P \,|\, P^{\leq k}   \isom N_{\lambda'}\} 
)$ in the proposition.

There are countably many finitely generated $\hat{R}$-modules, and we will show each infinite one occurs with probability $0$.  
Let $N$ be a finitely generated $\hat{R}$-module and define $\lambda(k)$ so that $N_{\lambda(k)'}=N^{\leq k}$.  In particular, 
if $N=\hat{R}^\ell \times T$ for some $T$ with $\mathfrak{m}^{k_0}T=0$, then for $k\geq k_0$,
we have that 
$\lambda(k)_j$ does not depend on $k$ for $j\leq k_0$,
and  $\lambda(k)_j=\ell$ for $k_0<j\leq k$.
So,  if $\ell\geq 1$, as $k\ra\infty$,  we have $
\nu(\{P \,|\, P^{\leq k}\isom N_{\lambda(k)'}\}
)\ra 0$, which implies $\nu(\{N\})=0$.  Thus the measure is supported on finite modules.  

If $\lambda'_1<k$, then the only finitely generated $\hat{R}$-module $P$ with $P^{\leq k} \isom N_{\lambda'}$ is $P=N_{\lambda'}$, which gives the final statement.
\end{proof}

We write $\Pf(t,m):=\Pf_q(t,m)$ when the value of $q$ is clear.  While we do not have a general product formula for $\Pf(t,m)$, we do have one when $t=0$.

\begin{lemma}\label{R:e0}
For $m$ a  non-negative integer,
$$
\Pf(0,m)=\prod_{j=1}^{\lceil m/2 \rceil} (1-q^{-(2j-1)}).
$$
\end{lemma}

\begin{proof}There are two forms of Pascal's identity for the $q$-binomial coefficients.

These are, for 
$1\leq e \leq m-1$,
  \[\frac{ \eta(m)}{ \eta(e) \eta(m-e)}  =  q^{-e} \frac{ \eta (m-1)}{ \eta(e) \eta(m-1-e)} +   \frac{ \eta (m-1)}{ \eta(e-1) \eta(m-e)} \] and  \[\frac{ \eta(m)}{ \eta(e) \eta(m-e)}  =  \frac{ \eta (m-1)}{ \eta(e) \eta(m-1-e)} +   q^{ - (m-e) } \frac{ \eta (m-1)}{ \eta(e-1) \eta(m-e)} .\] They can both be verified immediately by expanding the products. Plugged into \[\Pf(t, m) =\sum_{e=0}^m  (-1)^e q^{-  (t+1) e} \frac{ \eta(m)}{ \eta(e) \eta(m-e)} \] these give the recurrences
\[ \Pf(t,m)  = \Pf( t+1, m-1) -   q^{ - (t+1) } \Pf(t, m-1) \]
and
\[\Pf(t, m) = \Pf( t, m-1) -   q^{-m-t } \Pf( t-1, m-1) \]

Specializing we get
\begin{equation}\label{first-recurrence} \Pf(-1, m) = \Pf(0,m-1)  - \Pf(-1, m-1)\end{equation}
\begin{equation} \label{second-recurrence} \Pf(0, m) = \Pf(0,m-1) - q^{-m} \Pf(-1,m-1). \end{equation}

We have $\Pf(-1,m)=0$ for $m$ odd since the $e$ term cancels the $m-e$ term. That, together with \eqref{first-recurrence}, gives 
\[ \Pf(0,2k) = \Pf( -1, 2k) + \Pf(-1, 2k+1) = \Pf(-1, 2k) = \Pf(-1, 2k) +\Pf(-1, 2k-1)  = \Pf(0, 2k-1)\]
 which together with \eqref{second-recurrence} for $m=2k-1$ gives 
\begin{align*}
&\Pf(-1, 2k)=  \Pf(0,2k-1) =\Pf(0,2k-2) - q^{-(2k-1) } \Pf(-1,2k-2) \\= &\Pf(-1,2k-2) - q^{-(2k-1)} \Pf(-1, 2k-2)=(1- q^{-(2k-1)} ) \Pf(-1,2k-2)
\end{align*} 

so by induction and $\Pf(-1,0)=1$ we have \[\Pf(0,2k)= \Pf(0,2k-1) = \Pf(-1, 2k) =\prod_{j=1}^k ( 1- q^{-(2j-1)}) \] which gives the stated formula.\end{proof}

Before we consider our next moments, we verify that a certain $q$-series is always non-negative.

\begin{lemma}\label{L:Q}
For integers $\lambda_{-1}, \lambda, \phi_{-1}, \phi,s_1,s_2$,
with $\lambda_{-1}\geq \lambda$ and $\phi_{-1}\geq \phi$, 
 we define
\begin{align*} &\Qf_q(\lambda_{-1}, \lambda, \phi_{-1}, \phi,s_1,s_2):=\\
&\sum_{e=0}^{ \lambda_{-1}-\lambda} \sum_{f=0}^{ \phi_{-1}-\phi}
\frac{ (-1)^{e+f} q^{\binom{\lambda+e}{2}-\binom{\lambda}{2}-s_1(\lambda+e)
+\binom{\phi+f}{2}-\binom{\phi}{2}-s_2(\phi+f) 
 + (\lambda+e)(\phi+f)
}}
{q^{(\lambda+e)^2+(\phi+f)^2}  \eta(e)\eta(\lambda_{-1}-\lambda-e)
\eta(f)\eta(\phi_{-1}-\phi-f)}.
\end{align*}
If $\phi_{-1}-\lambda_{-1}\leq s_1$ and $\lambda_{-1} -\phi_{-1}\leq s_2$, then
\begin{enumerate}
\item $\Qf_q(\lambda_{-1}, \lambda, \phi_{-1}, \phi,s_1,s_2)\geq 0$, and
\item $\Qf_q(\lambda_{-1}, \lambda, \phi_{-1}, \phi,s_1,s_2)$ is positive if and only if
$\phi-\lambda\leq s_1$ and $\lambda -\phi\leq s_2$.
\end{enumerate}
\end{lemma}

\begin{proof}
We rearrange to obtain
\begin{align*} &\Qf_q(\lambda_{-1}, \lambda, \phi_{-1}, \phi,s_1,s_2)\\
&=\sum_{e=0}^{ \lambda_{-1}-\lambda} \sum_{f=0}^{ \phi_{-1}-\phi}
\frac{ (-1)^{e+f} q^{\binom{\lambda+e}{2}-\binom{\lambda}{2}-s_1(\lambda+e)
+\binom{\phi+f}{2}-\binom{\phi}{2}-s_2(\phi+f) 
 + (\lambda+e)(\phi+f)
}}
{q^{(\lambda+e)^2+(\phi+f)^2}  \eta(e)\eta(\lambda_{-1}-\lambda-e)
\eta(f)\eta(\phi_{-1}-\phi-f)}\\
&=
q^{-{\lambda^2}-s_1\lambda-\phi^2 -s_2\phi  +\lambda\phi
}
\sum_{e=0}^{ \lambda_{-1}-\lambda} \sum_{f=0}^{ \phi_{-1}-\phi}
\frac{ (-1)^{e+f} q^{
-\binom{e}{2} +e(-\lambda-s_1+\phi-1) 
-\binom{f}{2} +f(-\phi-1-s_2+\lambda+e)
}}
{ \eta(e)\eta(\lambda_{-1}-\lambda-e)
\eta(f)\eta(\phi_{-1}-\phi-f)}.
\end{align*}

Now we apply the $q$-binomial theorem to the sum in the $f$, and obtain that the above sum is equal to
\begin{align*}
q^{-{\lambda^2}-s_1\lambda-\phi^2 -s_2\phi  +\lambda\phi
}
\sum_{e=0}^{ \lambda_{-1}-\lambda} 
\frac{ (-1)^{e} q^{
-\binom{e}{2} +e(-\lambda-s_1+\phi-1) 
}}
{ \eta(e)\eta(\lambda_{-1}-\lambda-e)\eta(\phi_{-1}-\phi)
}
\prod_{i=0}^{ \phi_{-1}-\phi-1} (1-q^{-\phi-1-s_2+\lambda+e -i})
.
\end{align*}

If $e\geq s_2-\lambda+\phi+1,$ then the $i=-\phi-1-s_2+\lambda+e$ term in the product
exists because
$$
0 \leq -\phi-1-s_2+\lambda+e 
\leq -\phi-1-s_2+\lambda_{-1} 
\leq \phi_{-1}-\phi-1.
$$
This term is $(1-q^{0})=0$, and thus the summand for $e\geq s_2-\lambda+\phi+1$ is $0$.

If $\lambda-\phi>s_2$, in then follows that every summand is $0$
and
$\Qf_q(\lambda_{-1}, \lambda, \phi_{-1}, \phi,s_1,s_2)=0$.  Similarly, if 
$\phi-\lambda>s_1$ we also have $\Qf_q(\lambda_{-1}, \lambda, \phi_{-1}, \phi,s_1,s_2)=0$. 

For the rest of the proof we assume that $\phi-\lambda\leq s_1$ and $\lambda -\phi\leq s_2$, and we only consider summands with
$e\leq s_2-\lambda+\phi$ as the others are $0$.
In this case the exponents of $q$ in the product are  
$$
-\phi-1-s_2+\lambda+e -i
\leq -1
$$
and so the only non-positive factor in each summand is the $(-1)^{e}$.

If we let $a_{e}$ be the $e$ summand in our sum, then we can compute
$$
\frac{|a_{e+1}|}{|a_{e}|}=q^{-e -\lambda-s_1+\phi-1 }\frac{(1-q^{\lambda_{-1}-\lambda-e})
}{(1-q^{-e-1})}
\frac{(1-q^{-\phi-1-s_2+\lambda+e +1})}{(1-q^{-s_2+\lambda+e -\phi_{-1}})}.
$$
For $0\leq e \leq  \min(s_2-\lambda+\phi, \lambda_{-1}-\lambda)-1$, we then have
$$
\frac{|a_{e+1}|}{|a_{e}|} < \frac{q^{-e -\lambda-s_1+\phi-1 }}{(1-q^{-1})^2}.
$$
If $-e -\lambda-s_1+\phi-1 \leq -2$, then the above is $\leq 1$, then then our sum for $\Qf_q$ is an alternating sum with
first term positive and terms decreasing in absolute value, and thus (strictly) positive.
Otherwise, if $-e -\lambda-s_1+\phi-1 \geq -1$, then
$e \leq -\lambda-s_1+\phi\leq 0$, and so $-\lambda-s_1+\phi=0$.
Similarly,  $\Qf_q$ is positive or $-\phi-s_2+\lambda=0$.  
If $-\phi-s_2+\lambda=0$, then there is only $1$ summand  we consider, $e=0$, and thus $\Qf_q$ is positive.
\end{proof}

\begin{definition}\label{D:Qbar}
We define $\bar{\Qf}_q(\lambda_{-1}, \lambda, \phi_{-1}, \phi,s_1,s_2)$ so that
$$
\Qf_q(\lambda_{-1}, \lambda, \phi_{-1}, \phi,s_1,s_2)=\frac{q^{-{\lambda^2}-s_1\lambda-\phi^2 -s_2\phi  +\lambda\phi
}}{\eta(\lambda_{-1}-\lambda)\eta(\phi_{-1}-\phi)}\bar{\Qf}_q(\lambda_{-1}, \lambda, \phi_{-1}, \phi,s_1,s_2),
$$
where $\Qf_q$ is defined in Lemma~\ref{L:Q}.
\end{definition}

Now we consider the ring $R\times R$, whose modules we view as pairs of $R$-modules.

\begin{proposition}\label{P:DVRdualpairs}
Let $q,R$ be as above. Let $s_1,s_2$ be integers such that 
$s_1,s_2\geq 0$,
and let $r\geq 1$ be an integer.
Let $M_{N_{\rho'}, N_{\pi'}}=q^{\left(\sum_{j=1}^r \rho_j\pi_j \right)-s_1|\rho| -s_2|\pi|}$ for all partitions $\rho,\pi$.
  Then the values $M_N$
  are well-behaved for $R\times R$ and are the moments of a measure $\nu$ such that  for all $k_1,k_2\geq 1$
  and all partitions $\lambda,\phi$ with $\lambda_{k_1+1}=\phi_{k_2+1}=0$, we have
 \begin{align*}
 \nu(\{P \,|\, & P^{\leq \underline{k}}\isom(N_{\lambda'}, N_{\phi'}) \})=
\frac{q^{  
\left(\sum_{j=1}^r \lambda_j\phi_j \right)+\sum_j -(\lambda_j)^2 -s_1\lambda_j -(\phi_j)^2 -s_2\phi_j}}{\prod_{j\geq 2} \eta(\lambda_{j-1}-\lambda_j)  \eta(\phi_{j-1}-\phi_j) 
}
\\
& \times
\frac{\eta(\infty)\eta(s_1+s_2)}{\eta(s_2-\lambda_1+\phi_1)\eta(s_1+\lambda_1-\phi_1)}
\prod_{j=2}^{\min(r,k_1,k_2)} \bar{\Qf}_q(\lambda_{j-1}, \lambda_j, \phi_{j-1}, \phi_j,s_1,s_2)\\
&\times \prod_{i=  \lambda_{k_1}+1}^{  \lambda_{\min(r,k_1,k_2)} } (1 - q^{ - s_1- i } )
\prod_{i=  \phi_{k_2}+1}^{  \lambda_{\min(r,k_1,k_2)} } (1 - q^{ - s_2- i } )
 \end{align*}
 if  $-s_1\leq \lambda_{1} -\phi_{1}\leq s_2$, and  $\nu(\{P \,|\,  P^{\leq \underline{k}}\isom(N_{\lambda'}, N_{\phi'}) \})=0$ otherwise.
 The above expression is positive if and only if  $-s_1\leq \lambda_{j} -\phi_{j}\leq s_2$ for all $1\leq j \leq \min(r,k_1,k_2)$.
The measure $\nu$ is supported on pairs of finite $R$-modules.
\end{proposition}

Note the factor
$$
\frac{q^{  
\left(\sum_{j=1}^r \lambda_j\phi_j \right)+\sum_j -(\lambda_j)^2 -s_1\lambda_j -(\phi_j)^2 -s_2\phi_j}}{\prod_{j\geq 2} \eta(\lambda_{j-1}-\lambda_j)  \eta(\phi_{j-1}-\phi_j) 
}=\frac{M_{N_{\lambda'}, N_{\phi'}}}{|\Aut(N_{\lambda'}, N_{\phi'})|}.
$$

\begin{proof}
By considerations as in the proof of Proposition~\ref{P:DVRselfdual}  (in particular the second line of the large display) we have
\begin{align*}
&v_{\underline{k}, (N_{\lambda'},N_{\phi'})}
\\=&\sum_{\substack{e_1\dots,e_{k_1}\geq 0\\
f_1\dots,f_{k_2}\geq 0\\
\forall i \geq 1:e_{i+1}\leq \lambda_i-\lambda_{i+1}\\
\forall i \geq 1:f_{i+1}\leq \phi_i-\phi_{i+1}
}}  
\frac{ (-1)^{\sum_j e_j+f_j} q^{\sum_j \left(\binom{\lambda_j+e_j}{2}-\binom{\lambda_j}{2}-s_1(\lambda_j+e_j)
+\binom{\phi_j+f_j}{2}-\binom{\phi_j}{2}-s_2(\phi_j+f_j) 
\right) +\sum_{j=1}^{\min(r,k_1,k_2)} (\lambda_j+e_j)(\phi_j+f_j)
}
}
{q^{\sum_j (\lambda_j+e_j)^2+(\phi_j+f_j)^2} \prod_j  \eta(e_j)\eta(\lambda_j-\lambda_{j+1}-e_{j+1})
\eta(f_j)\eta(\phi_j-\phi_{j+1}-f_{j+1})}
\end{align*}
(By convention, $\lambda_j=e_j=0$ for $j>k_1$ and $\phi_j=f_j=0$ for $j>k_2$.)

If $k_1$ or $k_2=0$, then  $v_{\underline{k}, (N_{\lambda'},N_{\phi'})}$ can be computed just in the category of finite $R$-modules, and 
this case is treated in \cite[Lemma 6.3]{Sawin2022} for moments $M_N=|N|^{-s}$, and in particular $v_{\underline{k}, (N_{\lambda'},N_{\phi'})}\geq 0$.

As in the proof of Proposition~\ref{P:DVRselfdual}, the above sum factors over $j$, and the only factor that is an infinite sum is the $j=1$ sum, which is the same 
as the sum in Proposition~\ref{P:dualpairsvs} (the $\lambda,\phi,\trye,\tryf$ there are the $\lambda_1,\phi_1,e_1,f_1$ here,  respectively), as long as $k_1,k_2\geq 1$. 
 By the argument in the proof of
 Proposition~\ref{P:dualpairsvs} we then have that the above sum converges absolutely
 as long as $s_1+s_2+1>0$.  The $j=1$ factor 
$$
\sum_{\substack{e_1\geq 0\\
f_1\geq 0
}}  
\frac{ (-1)^{ e_1+f_1} q^{ \left(\binom{\lambda_1+e_1}{2}-\binom{\lambda_1}{2}-s_1(\lambda_1+e_1)
+\binom{\phi_1+f_1}{2}-\binom{\phi_1}{2}-s_2(\phi_1+f_1) 
\right) + (\lambda_1+e_1)(\phi_1+f_1)
}
}
{q^{ (\lambda_1+e_1)^2+(\phi_1+f_1)^2}   \eta(e_1)
\eta(f_1)}
$$ 
 evaluates to $\eta(\lambda_1)\eta(\phi_1)D_q(\lambda_1,\phi_1,s_1,s_2)$ as in the proof of Proposition~\ref{P:dualpairsvs}, and in particular is non-negative and is positive if and only if $\phi_{1}-\lambda_{1}\leq s_1$ and $\lambda_{1} -\phi_{1}\leq s_2$
 
The factor for $2\leq j \leq \min(r,k_1,k_2)$  is $\Qf_q(\lambda_{j-1}, \lambda_j, \phi_{j-1}, \phi_j,s_1,s_2)$, and Lemma~\ref{L:Q} addresses exactly when this is positive or $0$.  Inductively, it follows that
$$D_q(\lambda_1,\phi_1,s_1,s_2)\prod_{j=2}^{\min(r,k_1,k_2)} \Qf_q(\lambda_{j-1}, \lambda_j, \phi_{j-1}, \phi_j,s_1,s_2)\geq 0,$$ and is positive if and only
if  $\phi_{j}-\lambda_{j}\leq s_1$ and $\lambda_{j} -\phi_{j}\leq s_2$ for all $1\leq j \leq \min(r,k_1,k_2)$.

 The $\lambda$ part of the factor for $\min(r,k_1,k_2)< j \leq k_1$ further factors into
 \begin{align*}
 \sum_{e_j=0}^{ \lambda_{j-1}-\lambda_{j}} 
\frac{ (-1)^{ e_j} q^{\left(\binom{\lambda_j+e_j}{2}-\binom{\lambda_j}{2}-s_1(\lambda_j+e_j)
\right) 
}
}
{q^{(\lambda_j+e_j)^2} \eta(e_j)\eta(\lambda_{j-1}-\lambda_{j}-e_{j})
}
=&
q^{-{(\lambda_j)^2}-s_1\lambda_j
}
\sum_{e_j=0}^{ \lambda_{j-1}-\lambda_{j}} 
\frac{ (-1)^{e_j} q^{
-\binom{e_j}{2} +e_j(-\lambda_j-s_1+1) 
}}
{ \eta(e_j)\eta(\lambda_{j-1}-\lambda_{j}-e_{j})
}\\
=&
\frac{q^{-{(\lambda_j)^2}-s_1\lambda_j
}}{\eta(\lambda_{j-1}-\lambda_{j})}
\prod_{i=0}^{ \lambda_{j-1}-\lambda_{j}-1} (1-q^{-\lambda_j-s_1-1-i}) 
.
 \end{align*}
for each $j\leq k_1$, and an analogous expression for $j\leq k_2$ involving $\phi_j,\phi_{j-1},s_2$. 
 Since $s_i\geq 0$, we see the exponents of $q$ in the products above are all negative, and thus the above expression is always positive.

 We pull out all the powers of $q$ and expressions $\eta(\lambda_{j-1}-\lambda_{j})$ and $\eta(\phi_{j-1}-\phi_{j})$, and we 
 combine the remaining products for the $j>\min(r,k_1,k_2)$ terms as in the proof of Proposition~\ref{P:DVRselfdual}.
 Putting this all together, we obtain the formula for the measures.
 
 There are countably many finitely generated $\hat{R}\times\hat{R}$-modules, and we will show each infinite one occurs with probability $0$.  
Let $N=(N_1,N_2)$ be an ordered pair of finitely generated $\hat{R}$-modules. 
We define partitions $\lambda(k)$ and $\phi(k)$ such that $(N_{\lambda(k)'},N_{\phi(k)'})=N^{\leq \underline{k}}$.  In particular, 
if $N_i=\hat{R}^{\ell_i} \times T_i$ for some $T_i$ with $\mathfrak{m}^{k_0}T_i=0$, then for $k\geq k_0$,
we have that 
$\lambda(k)_j$ and $\phi(k)_j$ do not depend on $k$ for $j\leq k_0$,
and  $\lambda(k)_j=\ell_1$
and $\phi(k)_j=\ell_2$
 for $k_0<j\leq k$. If either $\ell_1\geq 1$ or $\ell_2\geq 1$, then
 $\nu(\{ P \,|\, P^{\leq \underline{k}}\isom N^{\leq\underline{k}}\}
)\ra 0$, which implies $\nu(\{N\})=0$.  Thus the measure is supported on pairs of finite modules.  
\end{proof}

\section{Representations of $\Gamma$ and formulas for our conjectured distributions}\label{S:Gamma}
Let $\Gamma$ be a finite group.
In this section, we explain the translation between the moment problem for $\Z_p[\Gamma]$-modules and the one for the kind of $R$ we considered in
Sections~\ref{S:mom}, \ref{S:vs} and \ref{S:DVR}.

For a prime $p\nmid |\Gamma|$, we let $T=\Z_p[\Gamma]$. By \cite[Theorem 74.11]{CurtisReiner2} and \cite[Proposition 27.1, Theorem 26.20(i), and Exercise 26.10]{CurtisReiner1}, we have $T \cong \prod_{i=1}^m  M_{n_i \times n_i} (R_{d_i})$ for some natural numbers $n_i$ and $d_i$, with $R_{d_i}$ the ring of integers in the degree $d_i$ unramified extension of $\mathbb Q_p$. Let $e_i$ be the $i$th primitive central idempotent of $T$, so that $e_i T \cong M_{n_i \times n_i} (R_{d_i})$. Then $e_i T$ has a unique indecomposable projective module $V_i \cong R_{d_i}^{n_i}$. Considering modules over $T \otimes \mathbb Q_p$ and $T\otimes \mathbb F_p$ respectively, the (distinct isomorphism classes of) irreducible representations of $\Gamma$ over $\Q_p$ are $V_i\tensor \Q_p$ and the (distinct isomorphism classes of) irreducible representations of 
$\Gamma$ over $\F_p$ are $V_i\tensor \F_p$. Note that $V_i$ is a $\mathbb Z_p$-module (and $V_i \tensor \mathbb Q_p$ and $V_i \tensor \F_p$ are representations) of rank $n_i d_i$ and $V_i \otimes \overline{\mathbb Q_p}$ and $V_i \otimes \overline{\mathbb F_p}$ split as a sum of $d_i$ irreducible representations. 

We always choose an ordering so that $V_1$ corresponds to the trivial representation. 

Since $T=\prod_{i=1}^m e_iT$, the category of finite $T$-modules is the product of the categories of finite $e_iT$ modules. By the Morita theorem, the category of $R_{d_i}$-modules  is equivalent to the category of $e_iT$-modules, by an equivalence $\mathcal{F}_i$ 
that takes an $R_{d_i}$-module $H$ to an $e_iT$-module whose underlying $R_{d_i}$-module is $H^{n_i}$.
Thus, the computations that we have done above in the categories of $R_{d_i}$-modules can be applied to $e_iT$-modules.

There is an involution
$\sigma$ on the primitive central idempotents $e_i$, fixing those corresponding to self-dual representations (over $\Q_p$ or $\F_p$), and exchanging 
the $e_i,e_{i'}$ corresponding to pairs of dual representations.
For an $e_i$ fixed by $\sigma$ corresponding to a representation $V_i$ over $\F_p$, 
let $\kappa_i=\F_{p^{d_i}}=\End_\Gamma(V_i)$, and
let $\epsilon_i=-1$ if $(\wedge_{\kappa_i}^2 V_i)^\Gamma\ne 0$, let 
$\epsilon_i=1$ if $(\wedge_{\kappa_i}^2 V_i)^\Gamma= 0$ but $(V_i\tensor_{\kappa_i}V_i)^\Gamma \ne 0$, and let  $\epsilon_i=0$ if $(V_i\tensor_{\kappa_i}V_i)^\Gamma= 0$.

The following lemmas are straightforward calculations from the definitions.

\begin{lemma}\label{L:selfdualmom}
Let $\Gamma, p, e_i,d_i,n_i,\epsilon_i$ be as above such that $e_i$ corresponds to a self-dual representation $V_i$ over $\F_p$.
Let $\lambda$ be a partition, and let $V_\lambda$ be the $e_iT$-module corresponding (via the Morita theorem) to the $R_{d_i}$-module $N_\lambda$.
Let $q=p^{d_i}$.  
Let $r\geq 0$ and $u$ be integers.
Then
$$
|(\wedge^2_{\Z_p} V_{\lambda'} [p^r])^\Gamma||V_{\lambda'}|^{-u}=q^{\sum_{j=1}^r 
\binom{\lambda_j}{2} 
+(1-\epsilon_i)\frac{\lambda_j}{2} +\sum_{j\geq 0} 
-un_i\lambda_j }.
$$
\end{lemma}

\begin{lemma}\label{L:nonselfdualmom}
Let $\Gamma, p, e_i,d_i,n_i$ be as above such that $e_i,e_{i'}$ correspond to non-isomorphic dual representations $V_i,V_{i'}$ over $\F_p$.
Let $\lambda,\phi$ be partitions, and let $V_\lambda$ be the $e_iT$-module corresponding (via the Morita theorem) to the $R_{d_i}$-module $N_\lambda$,
and $W_\phi$ be the $e_{i'}T$-module corresponding to the $R_{d_i}$-module $N_\phi$.
Let $q=p^{d_i}$.  
Let $r\geq 0$ and $u$ be integers.
Then
$$
|(\wedge^2_{\Z_p} (V_{\lambda'} \times W_{\phi'}) [p^r])^\Gamma||V_{\lambda'} \times W_{\phi'}|^{-u}=
q^{\sum_{j=1}^r \lambda_j\phi_j +\sum_{j\geq 0} 
-un_i\lambda_j -un_i\phi_j}.
$$
\end{lemma}

Let $\mathfrak{m}_i$ be the maximal ideal of $R_{d_i}$ (which is generated by $p$, but we write like this since we will consider modules for the product of $R_{d_i}$.) 
  For $\underline{k}=(k_2,\dots,k_m)\in \mathbb{N}^{m-1}$, we write
$\mathfrak{m}^{\underline{k}}=\prod_{i\geq 2} \mathfrak{m}_i^{k_i}$ for the ideal of $T$. 
The finite $T/e_1T$-modules are given by tuples $\underline{\lambda}=(\lambda^2,\dots,\lambda^m)$ of partitions
such that $\underline{\lambda}$ corresponds to the module $V_{\underline{\lambda}}:=\oplus_{i} \mathcal{F}_i(N_{\lambda^i})$ 
(and $N_{\lambda^i}$ is the $R_{d_i}$-module described in Section~\ref{S:mom}).\footnote{We apologize for the slight abuse of notation, since the ring $R$ is implicit in the notation $N_{\lambda}$. }
 For a $T/e_1T$-module $V$, we let $V^{\leq \underline{k}}=
V/\mathfrak{m}^{\underline{k}}V.$  
For a tuple $\underline{\lambda}$ of partitions, we write $\underline{\lambda}'$ for the tuple in which each partition is the conjugate of the corresponding partition from $\underline{\lambda}$.

We can now give an explicit description of the measure of Conjecture~\ref{C}.  Recall that $\Pf_q$ and $\bar{\Qf}_q$ are defined in Proposition~\ref{P:DVRselfdual} and Definition~\ref{D:Qbar} respectively. 

\begin{theorem}\label{T:formula}
Let $\Gamma$ be a finite group and $p$ a prime such that $p\nmid |\Gamma|$.  Let $m,T, e_i, n_i,d_i,\epsilon_i$ be as above.  Let $q_i=p^{d_i}$.
Let $r,u$ be positive integers.
Let $S=T/e_1T$.  Then there is a unique measure $\nu$ on the set
of isomorphism classes of finite $S$-modules (with the discrete topology and $\sigma$-algebra)
such that
$$
\int_X |\Sur(X,V)|d\nu=\frac{|(\wedge^2_{\Z_p}V)^\Gamma[p^r]|}{|V|^u}
$$
for every finite $S$-module $V$. 
For all $\underline{k}\in \mathbb{N}^{m-1}$ and tuples of partitions $\underline{\lambda}$ such that $\lambda^i_{k_i+1}=0$ for all $i\geq 2$,
\begin{align*}
&\nu(\{V \,|\,V^{\leq \underline{k}} \isom  V_{\underline{\lambda}'} \})=
\frac{|(\wedge^2_{\Z_p} V_{\underline{\lambda}'})^\Gamma[p^r]|}{|V_{\underline{\lambda}'}|^u|\Aut(V_{\underline{\lambda}'})|}\\
&\times
\prod_{\substack{i\geq 2\\\sigma(e_i)=e_i}} \left(
\prod_{\ell\geq 0}(1+ q_i^{-un_i-\frac{\epsilon_i+1}{2}-\ell})^{-1} 
\prod_{\ell= 2}^{\min(r,k_i) } 
\Pf_{q_i}(un_i-\frac{1-\epsilon_i}{2},\lambda^i_{\ell-1}-\lambda^i_\ell)
\prod_{\ell=  \lambda^i_{k_i}+1}^{  \lambda^i_{\min(r,k_i)} } (1 - q_i^{ - un_i- \ell } ) \right)\\
&\times
\prod_{\substack{2\leq i< j\\\sigma(e_i)=e_{j}}}  \left(
\frac{\eta_{q_i}(\infty)\eta_{q_i}(2un_i)}{\eta_{q_i}(un_i-\lambda^i_{1}+\lambda^{j}_{1})\eta_{q_i}(un_i+\lambda^i_{1}-\lambda^{j}_{1})}
\right.
\\
&\left. 
\prod_{\ell=2}^{\min(r,k_i,k_{j})} \bar{\Qf}_{q_i}(\lambda^{i}_{\ell-1},\lambda^{i}_{\ell}, \lambda^{j}_{\ell-1}, \lambda^{j}_{\ell},un_i,un_i)
 \prod_{\ell=  \lambda^i_{k_i}+1}^{  \lambda^i_{\min(r,k_i,k_{j})} } (1 - q_i^{ - un_i- \ell } )
\prod_{\ell=  \lambda^{j}_{k_{j}}+1}^{  \lambda^{j}_{\min(r,k_i,k_{j})} } (1 - q_i^{ - un_i- \ell } ) \right),
\end{align*}
if  
 $|\lambda^i_{1} -\lambda^{j}_{1}|\leq un_i$ for all $2\leq i<j$ with $\sigma(e_i)=e_j$,
and $\nu(\{V \,|\,V^{\leq \underline{k}} \isom  V_{\underline{\lambda}'} \})=0$ otherwise.
The expression above is positive 
if and only if  $|\lambda^i_{\ell} -\lambda^{j}_{\ell}|\leq un_i$ for all $2\leq i<j$ with $\sigma(e_i)=e_j$
and all $1\leq \ell\leq  \min(r,k_i,k_j)$.
Moreover, any sequence $\nu^t$ of measures of  isomorphism classes of finite $S$-modules such that
$$
\lim_{t\ra\infty} \int_{X} |\SSur(X,V)| d\nu^t= \frac{|(\wedge^2_{\Z_p}V)^\Gamma[p^r]|}{|V|^u}, 
$$
for every finite $S$-module $V$,
has the property that $\lim_{t\ra\infty} \nu^t(V)=\nu(V)$ for every finite $S$-module $V$.
\end{theorem}

If one wants a completely explicit formula,  $ |(\wedge^2_{\Z_p}V)^\Gamma[p^r]| $ is multiplicative over $\sigma$-orbits of the $e_i$, and
explicit formulas for each factor are given in Lemmas~\ref{L:selfdualmom} and \ref{L:nonselfdualmom}.
We also have $|\Aut(V_{\underline{\lambda}'})|=q^{\sum_{i,j} (\lambda^i_j)^2 }\prod_{i,j} \eta_{q_i}(\lambda^i_j-\lambda^i_{j+1} )$.

\begin{proof}[Proof of Theorem~\ref{T:formula}]
Using the Morita Theorem, we can translate this to an analogous statement on finite $\prod_{i\geq 2} R_{d_i}$-modules.  
We then apply Theorem~\ref{T:mom}.  Here note that the moments factor over the $\sigma$ orbits of $e_2,\dots e_m$.
Thus the sum defining well-behavedness and giving formulas for the measure similarly factors.  In each factor,
we have moments given by Lemma~\ref{L:selfdualmom} or Lemma~\ref{L:nonselfdualmom}, to which we can apply Proposition~\ref{P:DVRselfdual}
or Proposition~\ref{P:DVRdualpairs} to obtain not only that the moments are well-behaved, but that the resulting measures are supported on 
finite modules and given by the corresponding factor of the formula in this result.  
When we apply Proposition~\ref{P:DVRselfdual}, we 
do so with $s=un_i $ and $\epsilon=\frac{1-\epsilon_i}{2}$, and so we
use $u\geq 1$ so that $s-\epsilon\geq 0$.
Noting
that $n_i=n_{j}$, $q_i=q_{j}$ gives the result.
\end{proof}

We can prove Theorem~\ref{T:ptors} similarly. The formulas for
Theorem~\ref{T:ptors} can be obtained from Theorem~\ref{T:formula} by taking $\underline{k}= (1,1,\dots)$.

\begin{proof}[Proof of Theorem~\ref{T:ptors}]
The proof is very similar to that of Theorem~\ref{T:formula} except that we apply Theorem~\ref{T:mom} for the category of 
representations of $\Gamma$ over $\F_p$ instead of the category $\Z_p[\Gamma]$-modules.  
The well-behavedness required to apply Theorem~\ref{T:mom}  is just the $\underline{k}= (1,1,\dots)$ special case of the well-behavedness
used in the proof of Theorem~\ref{T:formula}.  

The formulas for the probabilities given by Theorem~\ref{T:mom} are then precisely the $\underline{k}=(1,1,\dots)$ formulas from
Theorem~\ref{T:formula}. This simplifies the formulas substantially.  Each of the partitions $\lambda^i$ become integers $\lambda^i_1$.  For $\underline{\lambda}'$ such that $(\lambda')^i=(1,\dots,1)$ for each $i$ and 
($\lambda^i_1$ gives the number of $1$), we have
\begin{align*}
&\nu(\{V \,|\,V/pV \isom  V_{\underline{\lambda}'} \})=
\frac{|(\wedge^2_{\Z_p} V_{\underline{\lambda}'})^\Gamma[p]|}{|V_{\underline{\lambda}'}|^u|\Aut(V_{\underline{\lambda}'})|}\times
\\
&\prod_{\substack{i\geq 2\\\sigma(e_i)=e_i}} 
\prod_{\ell\geq 0}(1+ q_i^{-un_i-\frac{\epsilon_i+1}{2}-\ell})^{-1} 
\prod_{\substack{2\leq i< j\\\sigma(e_i)=e_{j}}}  \left(
\frac{\eta_{q_i}(\infty)\eta_{q_i}(2un_i)}{\eta_{q_i}(un_i-\lambda^i_{1}+\lambda^{j}_{1})\eta_{q_i}(un_i+\lambda^i_{1}-\lambda^{j}_{1})}
\right).
\end{align*}
if  $|\lambda^i_{1} -\lambda^{j}_{1}|\leq un_i$ for all $2\leq i<j$ with $\sigma(e_i)=e_j$ 
and $\nu(\{V \,|\,V/pV \isom  V_{\underline{\lambda}'} \})=0$ otherwise.
\end{proof}

Similarly, to obtain an analogue of Theorem~\ref{T:ptors} describing the $p^k$-torsion, we would take $\underline{k}= (k,k,\dots)$.

Since a finite $S$-module can be detected mod $\mathfrak{m}^{\underline{k}}$ for sufficiently large $k_i$, we can use Theorem~\ref{T:formula} to give the measure on any particular $S$-module.

\begin{corollary}\label{C:specific}
For the measure $\nu$ in Theorem~\ref{T:formula} and any tuple $\underline{\lambda}$ of partitions, we have
\begin{align*}
&\nu(\{ V_{\underline{\lambda}'} \})=
\frac{|(\wedge^2_{\Z_p} V_{\underline{\lambda}'})^\Gamma[p^r]|}{|V_{\underline{\lambda}'}|^u|\Aut(V_{\underline{\lambda}'})|}\\
&\times
\prod_{\substack{i\geq 2\\\sigma(e_i)=e_i}} \left(
\prod_{\ell\geq 0}(1+ q_i^{-un_i-\frac{\epsilon_i+1}{2}-\ell})^{-1} 
\prod_{\ell= 2}^{r } 
\Pf_{q_i}(un_i-\frac{1-\epsilon_i}{2},\lambda^i_{\ell-1}-\lambda^i_\ell)
\prod_{\ell=  1}^{  \lambda^i_{r} } (1 - q_i^{ - un_i -  \ell } ) \right)\\
&\times
\prod_{\substack{2\leq i< j\\\sigma(e_i)=e_{j}}}  \left(
\frac{\eta_{q_i}(\infty)\eta_{q_i}(2un_i)}{\eta_{q_i}(un_i-\lambda^i_{1}+\lambda^{j}_{1})\eta_{q_i}(un_i+\lambda^i_{1}-\lambda^{j}_{1})}
\right.
\\
&\left. 
\prod_{\ell=2}^{r} \bar{\Qf}_{q_i}(\lambda^{i}_{\ell-1},\lambda^{i}_{\ell}, \lambda^{j}_{\ell-1}, \lambda^{j}_{\ell},un_i,un_i)
 \prod_{\ell= 1}^{  \lambda^i_{r} } (1 - q_i^{ - un_i- \ell } )
\prod_{\ell=  1}^{  \lambda^{j}_{r} } (1 - q_i^{ - un_i- \ell } ) \right),
\end{align*}
if  
 $|\lambda^i_{1} -\lambda^{j}_{1}|\leq un_i$ for all $2\leq i<j$ with $\sigma(e_i)=e_j$,
and $\nu(\{ V_{\underline{\lambda}'} \})=0$ otherwise.
The expression above is positive 
if and only if  $|\lambda^i_{\ell} -\lambda^{j}_{\ell}|\leq un_i$ for all $2\leq i<j$ with $\sigma(e_i)=e_j$
and all $1\leq \ell\leq r$.
\end{corollary}
\begin{proof}
We apply Theorem~\ref{T:formula} with $k_i$ sufficiently large such that $\{V \,|\,V^{\leq \underline{k}} \isom  V_{\underline{\lambda}'} \}=\{ V_{\underline{\lambda}'} \}$ (so in particular, $\lambda^i_{k_i}=0$)
and also such that $k_i\geq r$.
\end{proof}

We now have a function field corollary of Theorem~\ref{T:MainFF} on the distribution of class groups of function fields.
Recall $E_{\Gamma}(n,q)$ is the set of isomorphism classes of extensions $K/\F_q(t)$ with a choice of isomorphism $\Gal(K/\F_q(t))\isom \Gamma$, that are split completely above $\infty$ and such that the radical of the discriminant ideal $\Disc(K/\F_q(t))$ has norm $q^n$.
For a finite set of primes $P$,  an abelian group $A$, let $A_P$ be the product of the Sylow $p$-subgroups of $A$ for $p\in P$.
An abelian \emph{$P$-group} is a finite product of abelian $p$-groups for $p\in P$.

\begin{corollary}\label{C:MainFF}
Let $\Gamma$ be a finite group and $P$ a finite set of primes not dividing $|\Gamma|$.
Let $H$ be a finite abelian $P$-group with an action of $\Gamma$ such that  $H^\Gamma=1$.  
For each $p\in P$, let $r_p$ be a positive integer.
Let $q_1,\dots$ be any sequence of  prime powers, all relatively prime to $|\Gamma|$ and all primes in $P$, and such that for all $i$, we have that $\F_{q_i}(t)$ contains the $p^{r_p}$th roots of unity, but not the $p^{{r_p}+1}$st roots of unity.

Then, as long as $q_x$ grows sufficiently fast with $x$, we have
$$
\lim_{x\ra\infty} 
\frac{\sum_{n\leq x} |\{K\in E_\Gamma(n,q_x) \,|\,  \Cl(\O_K)_P\isom_\Gamma H)  \} }{\sum_{n\leq x} |E_\Gamma(n,q_x)|}
 =\prod_{p\in P} \nu_p(H[ p^\infty] )
,$$
where $\nu_p$ is the measure from Theorem~\ref{T:formula} (given $\Gamma,p,r=r_p$ and $u=1$).
\end{corollary}

\begin{proof}
We start with Theorem~\ref{T:MainFF}.  It is a formal consequence  from the statement of Theorem~\ref{T:MainFF} with two limits
to obtain the same statement with just $\lim_{x\ra\infty}$ and $q$ replaced by $q_x$ (growing sufficiently fast with $x$).  
Now we apply Theorem~\ref{T:mom}, and in particular uniqueness and robustness.  The proof of Theorem~\ref{T:formula} shows that these moments are well-behaved.
Thus the corollary follows from Theorem~\ref{T:mom}.
\end{proof}

\section{Class groups cannot be certain modules}\label{S:dontappear}

Using Theorem~\ref{T:formula}, Conjecture~\ref{C} predicts that the density of fields for which $V_{\underline{\lambda}}$ is the class group is zero unless $ |\lambda_\ell^i - \lambda_\ell^j| \leq u n_i$ for all $2 \leq i<j$ with $\sigma(e_i)=e_j$, and $1 \leq \ell \leq r$. We can verify this prediction by showing that, in fact, fields violating this condition do not exist.

\begin{theorem}\label{T:da}
 Let $\Gamma$ be a finite group and $p$ a prime such that $p\nmid |\Gamma|$. Let $m, T, e_i, n_i ,d_i $ be as at the start of Section~\ref{S:Gamma}. Let $r$ and $u$ be positive integers. 

Let $K_0$ be a number field with $u$ infinite places containing the $p^r$th roots of unity. Let $K$ be a Galois extension of $K_0$ with Galois group $\Gamma$. Let $\underline{\lambda}$ be the unique tuple of partitions such that  $\Cl_K [p^\infty]/ \Cl_{K_0}[p^\infty] \cong V_{\underline{\lambda}'}$. 

Then $|\lambda_\ell^i - \lambda_\ell^j| \leq u n_i$ for all $2 \leq i<j$ with $\sigma(e_i)=e_j$, and $1 \leq \ell \leq r$.

\end{theorem}

The case $r=1$ of Theorem~\ref{T:da} was shown previously by Gras \cite[Theorem 7.7($\alpha$)(i) and Theorem 8.8($\alpha$)(i)]{Gras1998}, with the special case where $K$ is Galois over $\mathbb Q$ obtained earlier by Leopoldt \cite[Satz 2 on p. 170]{Leopoldt1958}.
See also \cite[Proposition 4.2.2]{Breen2021} for the case $K_0=\Q$ and $p=2$.

\begin{proof} As in \cite[Definition 4.1]{Lipnowski2020}, we can define a map $\psi_K$ \[\mathrm{Cl}(K)^\vee[{p^r}] \cong H^1(\operatorname{Spec}\mathcal O_K,\mathbb Z/{p^r}\mathbb Z) \rightarrow H^1(\operatorname{Spec}\mathcal O_K,\mu_{p^r}) \rightarrow \mathrm{Cl}(K)[{p^r}]\] where the first isomorphism arises from class field theory, the middle arrow is cup product with a generator of the $p^r$th roots of unity of $K_0$, and the last arrow arises from the Kummer exact sequence and the identification  $\mathrm{Cl}(K) \cong  H^1( \operatorname{Spec}\mathcal O_K, \mathbb G_m)$. Each step is clearly $\Gamma$-invariant (because, in particular, the roots of unity of $K_0$ are $\Gamma$-invariant) so this gives a map of $\mathbb Z-p[\Gamma]$-modules.

By the argument in the first two paragraphs of the proof of \cite[Lemma 6.21]{Lipnowski2020}, the kernel of $\psi_K$ is isomorphic as a $\mathbb Z_p[\Gamma]$-module to a submodule of $\mathcal O_K^\times \otimes \mathbb Z/p^r$.

Let $2 \leq i<j$ with $\sigma(e_i)=e_j$.
Applying the left-exact functor $\Hom_\Gamma (V_i, \cdot)$, we see that there is a map \[\Hom_\Gamma (V_i, \mathrm{Cl}(K)^\vee) [{p^r}] \to \Hom_\Gamma (V_i,  \mathrm{Cl}(K))[{p^r}]\] whose kernel is a subgroup of $\Hom_\Gamma (V_i, \mathcal O_K^\times  \otimes \mathbb Z/p^r)$.  We have $\Hom_\Gamma (V_i,  \mathrm{Cl}(K)) \cong \prod_{\ell=1}^{\infty}  (R_{d_i} /p^\ell)^{\lambda^i_\ell - \lambda^{i}_{\ell+1}} $ so $\Hom_\Gamma (V_i,  \mathrm{Cl}(K)) [p^r] \cong \prod_{\ell=1}^{r-1}   (R_{d_i}/p^\ell)^{\lambda^i_\ell - \lambda^{i}_{\ell+1}} \times (R_{d_i}/p^r) ^{\lambda^i_r} $. Similarly $\Hom_\Gamma (V_i, \mathrm{Cl}(K)^\vee)  \cong  \prod_{\ell=1}^{\infty}  (R_{d_i} /p^\ell)^{\lambda^j_\ell - \lambda^{j}_{\ell+1}} $ so $\Hom_\Gamma (V_i,  \mathrm{Cl}(K)^\vee) [p^r] \cong \prod_{\ell=1}^{r-1}   (R_{d_i}/p^\ell)^{\lambda^j_\ell - \lambda^{j}_{\ell+1}} \times (R_{d_i}/p^r) ^{\lambda^j_r} $.

Since our assumptions imply $K/K_0$ is split at $\infty$, each infinite place of $K_0$ lies under a simply transitive $\Gamma$-orbit of infinite places of $K$. Thus, by Dirichlet's unit theorem, $\mathcal O_K^\times \otimes \mathbb R$ is isomorphic to $ \mathbb R[\Gamma]^u / \mathbb R$ as an $\R[\Gamma]$-module. Since $p$ is prime to the order of $\Gamma$, this implies that $\mathcal O_K^\times \otimes \mathbb Z_p$, modulo torsion, is isomorphic to $\mathbb Z_p[\Gamma]^u/ \mathbb Z_p$ as a $\Z_p[\Gamma]$-module.  
Further, using $p\nmid |\Gamma|$, we then have that, as a $\Z_p[\Gamma]$-module,  $\O_K^\times\tensor \Z_p\isom \mathbb Z_p[\Gamma]^u/ \mathbb Z_p \times
(\mu_{\infty} (K) \tensor\Z_p)$.  Since the $p$th roots of unity are in $K_0$, the cyclic group $\mu_{\infty} (K) \tensor\Z_p$ has a trivial $\Gamma$-submodule, and hence
has trivial $\Gamma$-action since $p\nmid |\Gamma|$.  
Thus $\Hom (V_i, \mathcal O_K^\times \otimes \mathbb Z/p^r)  \cong  (R_{d_i} /p^r)^{u n_i}$.

It follows that we have a map \[  \prod_{\ell=1}^{r-1}   (R_{d_i}/p^\ell)^{\lambda^j_\ell - \lambda^{j}_{\ell+1}} \times (R_{d_i}/p^r) ^{\lambda^j_r}  \ra \prod_{\ell=1}^{r-1}   (R_{d_i}/p^\ell)^{\lambda^i_\ell - \lambda^{i}_{\ell+1}} \times (R_{d_i}/p^r) ^{\lambda^i_r} \] whose kernel has $R_{d_i}/p$-rank $\leq  u n_i$. 
We say the $R_{d_i}$-module $\prod_{\ell\geq 1}  (R_{d_i}/p^\ell)^{\rho_\ell}$ has \emph{$R_{d_i}/p^\ell$-rank} $\rho_\ell+\rho_{\ell+1}+\cdots$.
Since in a map of $p$-group $R_{d_i}$-modules, the $R_{d_i}/p^\ell$-rank of the source minus the $R_{d_i}/p^\ell$ rank of the target is at most the $R_{d_i}/p$-rank of the kernel
(e.g.  see \cite[II (4.3)(i)]{Macdonald2015}), this gives  $\lambda^j_\ell - \lambda^i_\ell \leq  u n_i$ for each $1\leq \ell\leq r$  A symmetric argument gives $\lambda^i_\ell - \lambda^j_\ell\leq un_i$.
\end{proof}

\section{Non-Galois extensions}\label{S:NG}
In this section, we explain how our conjecture also gives, as a consequence, the distribution of Sylow $p$-subgroups of class groups of non-Galois fields in the presence of roots of unity.

Let $\Gamma$ be a finite group and $\Gamma'$ a subgroup of $\Gamma$.
If $K_0$ is a number field, and $K/K_0$ is an extension with an isomorphism $\Gal(K/K_0)\isom \Gamma$,  we can also consider the (relative) class group $\Cl_{L|K_0}$ of the fixed field $L=K^{\Gamma'}$. Let $p$ be a prime not dividing $|\Gamma|$, and let 
$V_i$, $e_i$,$n_i$,$d_i$,$T$ be as at the start of Section~\ref{S:Gamma}., and let $S=T/e_1T$.  Let $m_id_i=\operatorname{rank}_{\Z_p} V_i^{\Gamma'}$.

Since $p\nmid |\Gamma|$, we have that the natural inclusion map
$$
\Cl_{L|K_0}[p^\infty] \ra \Cl_{K|K_0}[p^\infty]^{\Gamma'}
$$
is an isomorphism \cite[Corollary 7.7]{Cohen1990}.  
Thus, a conjecture on the distribution of $\Cl_{K|K_0}[p^\infty]$ as $\Z_p[\Gamma]$-modules formally implies a conjecture
on the distribution of $\Cl_{L|K_0}[p^\infty]$ as $\Z_p$-modules.  However, we will see that one can make the implication much more explicit.

Let $e_{\Gamma'}=|\Gamma'|^{-1} \sum_{\gamma\in \Gamma'} \gamma$.  
We have that $e_{\Gamma'}$ is a (not necessarily central) idempotent of $\Z_p[\Gamma]$.
For any $\Z_p[\Gamma]$-module $V$, we have that $e_{\Gamma'}V=V^{\Gamma'}$.
Let $\mathfrak{o}:= e_{\Gamma'}Se_{\Gamma'}.$  For any $S$-module $V$, we have that $V^{\Gamma'}$
is naturally an $\mathfrak{o}$-module, allowing the action to take place in $V$.

Let $e_{\Gamma|\Gamma'}$ be the central idempotent that is the sum of all the idempotents $e_i$ corresponding to non-trivial projective $\Z_p[\Gamma]$-modules that have non-trivial $\Gamma'$-invariants, i.e. the $e_i$ such that $e_{\Gamma'}e_i\ne 0$.
Let $\mathcal{M}$ be the set of the corresponding $i$.
The  $V_i$ for $i\in\mathcal{M}$ are exactly the non-trivial modules contained in $\Ind_{\Gamma'}^\Gamma \Z_p$ \cite[Lemma 7.10]{Wang2021}.
For any $S$-module $V$, we have that $V^{\Gamma'}=(e_{\Gamma|\Gamma'} V)^{\Gamma'}$.
We have $e_{\Gamma'}e_{\Gamma|\Gamma'}=e_{\Gamma'}$.

The ring $\mathfrak{o}$ is a maximal order in a semi-simple $\Q_p$ algebra with primitive central idempotents $ e_{\Gamma'}e_i=e_ie_{\Gamma'}$ for $i\in \mathcal{M}$ (over $\Q$, this can be found in \cite[Proposition 8.3, Corollary 8.10]{Wang2021}, and the argument over $\Q_p$ is the same). 
There is a Morita equivalence between $e_{\Gamma|\Gamma'}S$-modules and $\mathfrak{o}$-modules taking
$V\mapsto V^{\Gamma'}$ (see \cite[Theorem 8.9]{Wang2021}).  The inverse functor takes $U$ to $
\mathcal{I}(U):=e_{\Gamma|\Gamma'}S  e_{\Gamma'}\tensor_\mathfrak{o} U$. 
 If $U=\oplus_{i\in \mathcal{M}} e_{\Gamma'}e_iU$
is an $\mathfrak{o}$-module, 
and $V$ is a   $e_{\Gamma|\Gamma'}S$-module with $V^{\Gamma'}=U$ (so $V\isom \mathcal{I}(U)$), then
$V= \oplus_{i\in \mathcal{M}} e_iV$, with 
\begin{equation}\label{E:sizechange}
|e_iV|=|e_{\Gamma'}e_iU|^{n_i/m_i}
\end{equation}
by \cite[Theorem 7.3]{Cohen1990}.
The finite $\mathfrak{o}$-modules are thus indexed by tuples of partitions $(\lambda^i)_{i\in \mathcal{M}}$, given as $(V_{\underline{\lambda}})^{\Gamma'}$.
If $\underline{\lambda}$ is a tuple of partitions indexed by elements of $\mathcal{M}$, by abuse of notation, we also
use $\underline{\lambda}$ to denote the extension by zero, i.e. the tuple $(\lambda^2,\dots,\lambda^m)$, where $\lambda^i=0$ for $i\not\in \mathcal{M}$.
The following is an immediate consequence of the Morita equivalence.

\begin{proposition}\label{P:NG}
The measure $\nu$ on $S$-modules from Theorem~\ref{T:formula}, pushed forward to a measure on finite $\mathfrak{o}$-modules
via the map $V\mapsto V^{\Gamma'}$ from $S$-modules to $\mathfrak{o}$-modules, gives a measure
$\nu_{\Gamma'}$ on $\mathfrak{o}$-modules as follows.  Let $\underline{k}\in \N^{\mathcal{M}}$ and let
 $\underline{\lambda}$ be a tuple of partitions indexed by the elements of $\mathcal{M}$
such that $\lambda^i_{k_i+1}=0$ for all $i\in\mathcal{M}$.  Then
\begin{align}\label{E:ng}
\nu_{\Gamma'}(\{U  \,|\,\mathcal{I}(U)^{\leq \underline{k}} \isom  V_{\underline{\lambda}'} \})
=
\nu(\{V \,|\,V^{\leq \underline{k}} \isom  V_{\underline{\lambda}'} \}).
\end{align}
Also, the $U$-moment of $\nu_{\Gamma'}$ is $|(\wedge^2_{\Z_p} V)^\Gamma[p^r]||V|^{-u}$ for $V=\mathcal{I}(U)$.
\end{proposition}
The right-hand side of \eqref{E:ng} is given by an explicit formula in Theorem~\ref{T:formula} (this is just the special case of that formula when $k_i=0$ for $i \not\in \mathcal{M}$).  It then follows that the probability of an individual module $\nu_{\Gamma'}(\{ (V_{\underline{\lambda}'})^{\Gamma'}  \})$
is given by a formula that is
 the same as that given in Corollary~\ref{C:specific}, but with the product restricted to $i\in\mathcal{M}$.
Note that the terms
$$
\frac{|(\wedge^2_{\Z_p} V_{\underline{\lambda}'})^\Gamma[p^r]|}{|V_{\underline{\lambda}'}|^u|\Aut(V_{\underline{\lambda}'})|}
$$
 can all be given by explicit formulas (e.g. using Lemmas~\ref{L:selfdualmom} and \ref{L:nonselfdualmom}).
Also,  by the categorical equivalence, we have $|\Aut(V_{\underline{\lambda}'})|=|\Aut_{\mathfrak{o}}((V_{\underline{\lambda}'})^{\Gamma'})|.$

This shows that Conjecture~\ref{C} gives explicit conjectures for the distribution of Sylow $p$-subgroups of class groups of non-Galois fields in the presence of roots of unity,  for $p$ not dividing the order of the Galois closure, in particular that they are given by the measure $\nu_{\Gamma'}$ above.

\section{Experimental evidence}\label{S:data}

Malle \cite{Malle2008,Malle2010} has done extensive computations on class groups over extensions of a base field with roots of unity.
Most of his tables of computed class group data strongly support his conjectures, showing remarkably good convergence of the probabilities of various Sylow $p$-subgroups to the conjectured amount.  Since we see in Section~\ref{S:MalleCompare} 
that his conjectures are (in all cases we shall mention in this section) special cases of our conjecture, these all provide good evidence for our conjecture.  The cases for which this provides data are as follows.

\begin{enumerate}
\item $\Gamma=C_2$, $p=3$, $r=1,u=1$ $(K_0=\Q(\sqrt{-3}))$: \cite[Table 9]{Malle2008} and \cite[Table 1]{Malle2010} on $3$-ranks,  \cite[Table 2]{Malle2010} on $3$-Sylow

\item $\Gamma=C_2$, $p=5$, $r=1,u=1$ $(K_0=\Q(\zeta_5))$: \cite[Table 3]{Malle2010} on 5-ranks

\item $\Gamma=C_3$, $p=2$, $r=1,u=1$ $(K_0=\Q,\Q(\sqrt{-3}))$: \cite[Tables 3,4]{Malle2008} 
on $2$-ranks and \cite[Tables 5,6,7]{Malle2010} and \cite[Tables 7,8,9]{Malle2010} on $2$-Sylow

\item $\Gamma=C_3$, $p=2$, $r=1,u=2$ $(K_0=\Q(\sqrt{5}))$: \cite[Table 10]{Malle2010}
on $2$-ranks and \cite[Table 11]{Malle2010} on $2$-Sylow

\item $\Gamma=C_3$, $p=2$, $r=2,u=1$ $(K_0=\Q(i))$: \cite[Table 12]{Malle2010}
on $2$-ranks

\item  $\Gamma=C_5$, $p=2$, $r=1,u=1$ $(K_0=\Q)$: \cite[Table 8]{Malle2008} 
on $2$-ranks
\end{enumerate}

Further, Malle provides a chart of data for Sylow $2$-subgroups of class groups of $C_3$-extensions of $\Q(i)$ 
(so $r=2,u=1$)
for which he makes no conjecture to explain, only remarking they do not seem to fit well with any known formulas.
Our conjecture provides an excellent match for this data, which we present in Table~\ref{Table:Malle}.

\begin{table}[h!]
\centering
\begin{tabular}{|c c c c c c c c c c|} 
 \hline
 $(\lambda_1,\dots)$: & (0) & (1) & (1,1) & (2) & (1,1,1) & (2,1) & (3) & (1,1,1,1) & (2,1,1)  \\ [0.5ex] 
 D & 1 & $2^2$ & $4^2$ & $2^4$ & $8^2$ & $4^2\times 2^2$ & $2^6$ & $16^2$ & $8^2 \times 2^2$ \\ [0.5ex] 
 \hline\hline
$ \leq 10^{15}$&0.864&0.115&0.016&0.25E-2&0.12E-2&0.50E-3&0&0.7E-4&0.6E-4  \\ 
\hline
 $ \geq 10^{16}$ &0.859&0.120&0.016&0.30E-2&0.13E-2&0.42E-3&0&0.2E-3&0.4E-4\\
 $ \geq 10^{24}$ &0.854&0.123&0.017&0.40E-2&0.11E-2&0.64E-3&0.1E-4&0.7E-4&0.3E-4 \\
 $ \geq 10^{32}$ &0.853&0.125&0.017&0.40E-2&0.10E-2&0.62E-3&0.4E-4&0.6E-4&0.4E-4 \\
  \hline\hline
 Conj.~\ref{C} & 0.853 & 0.124 & 0.0166& 0.38E-2 & 0.10E-2 & 0.60E-3 & 0.38E-4 & 0.66E-4 & 0.37E-4  \\ 
    \cite[(6-2)]{Malle2010} &0.853&0.133&0.0083&0.44E-2&0.52E-3&0.34E-3&0.35E-4&0.33E-4&0.21E-4  \\ 
Cohen-Martinet &0.918&0.076&0.0048&0.03E-2&0.30E-3&0.25E-4&0.79E-7&0.19E-4&0.16E-4 \\ [1ex] 
 \hline
\end{tabular}
\caption{Malle's data on $2$-Sylow for $C_3$ extensions $K/\Q(i)$ from \cite[Table 13]{Malle2010} compared to various conjectures}
\label{Table:Malle}
\end{table}

If $K/\Q(i)$ is a $C_3$-extension, then $\Cl_{K|\Q(i)}[2^\infty]$ is a module for $\Z_2[\zeta_3]$,
where $\zeta_3$ is a primitive $3$rd root of unity.  
In Table~\ref{Table:Malle},  the columns correspond to the first few  finite $\Z_2[\zeta_3]$-modules,
with the top row labeling them in our notation assuming the module is $V_{\lambda'}$, and the second row labeling them with the notation of \cite{Malle2010}
which reflects their underlying structure as abelian groups.
Malle computed all $C_3$ extensions of $\Q(i)$ with discriminant at most $10^{15}$, and
for some other values of $D$ the first $S$ such extensions with discriminant at least $D$
 (it appears $S$ is $5\cdot 10^5$ or
$4\cdot 10^5$ in each case).  The numbers in the middle part of the table show, for each collection of extensions $K/\Q(i)$, the proportion  
with $\Cl_{K|\Q(i)}[2^\infty]$ which is the given $\Z_2[\zeta_3]$-module.
In the bottom part of the chart, one sees the proportions predicted 
by Conjecture \ref{C}, Malle's formula \cite[(6-2)]{Malle2010}, and the original conjecture of Cohen and Martinet \cite{Cohen1990}.
We see quite good agreement with Conjecture \ref{C}.  Malle recognized that the formula \cite[(6-2)]{Malle2010} was not a good match for these values, but he included this comparison so we do as well.

It would be very interesting to have more experimental evidence for Conjecture~\ref{C}.  The case of $\Gamma=C_7$, with $p=2$, and $K_0=\Q$ is particularly interesting, as it is one of the smallest cases (being about fields of absolute degree $7$) in which we have no current experimental evidence, and also the first in which there are dual pairs of representations, and so our conjecture predicts a lack of independence between representations (a new phenomenon worthy of backing up with experimental evidence).
Also $\Gamma=C_7\rtimes C_3$ has a transitive degree $7$ permutation action, and letting $\Gamma'$ be the stabilizer of a point, one could consider
the non-Galois $K$ that are the $\Gamma'$ fixed fields of $\Gamma$-extensions, with $p=2$, and $K_0=\Q$, to check a non-Galois implication of our conjecture for degree $7$ fields.
Other relatively small degree cases are degree $4$ abelian or $D_4$ extensions with
$p=3$ and $K=\Q(\zeta_3)$.
The case of $\Gamma=C_{11}$, with $p=2$, and $K_0=\Q$, is also of interest, because as we will see below in Section~\ref{S:ip}, our conjecture differs from a suggestion of Malle's in this case. Other interesting cases include $K_0=\mathbb Q(\zeta_5)$ with $\Gamma =C_4$ and $p=5$,  or $K_0=\mathbb Q(\zeta_3)$ with $\Gamma = C_8$ and $p=3$, as they have dual irreducible representations and $p$ odd.

\section{Comparison to previous conjectures in special cases}\label{S:MalleCompare}

Malle \cite{Malle2008,Malle2010} and Adam and Malle \cite{Adam2015} have made conjectures about class group distributions in the presence of roots of unity in particular situations.  In this section,  we explain how these conjectures relate to Conjecture~\ref{C}.

Let $\Gamma$ be a finite group and $\Gamma'$ be a subgroup with $\bigcap_{\gamma\in \Gamma} \gamma \Gamma' \gamma^{-1}=1.$
Let $p$ be a prime $p\nmid|\Gamma|$.
Throughout this section, we will consider the distribution of $\Cl_{K^{\Gamma'}/K_0}[p^\infty]$, as 
$K$ ranges over $\Gamma$ extensions of $K_0$ for some fixed number field $K_0$
containing $p^r$th roots of unity but not $p^{r+1}$th roots of unity, and with $u$ infinite places.
(So we can take $\Gamma'=1$ to get the case of relative class groups of Galois extensions, or other $\Gamma'$ to consider non-Galois extensions.)
We say $(\Gamma,\Gamma')$ is an \emph{absolutely irreducible pair} if the representation $\Ind_{\Gamma'}^\Gamma \C$ is the sum of one trivial representation and an irreducible representation.

The conjectures of Malle and Adam are mostly in the case when $(\Gamma,\Gamma')$ is an absolutely irreducible pair.  We have that $(S_n,S_{n-1})$ is an absolutely irreducible pair, so this covers the case of degree $n$ extensions with Galois closure with Galois group the symmetric group $S_n$, including quadratic extensions. 
However,  such pairs are rare, with only 14 examples with $|\Gamma|\leq 120$, and the only example giving Galois fields (i.e. with $\Gamma'=1$) being $\Gamma=S_2$.

When $(\Gamma,\Gamma')$ is an absolutely irreducible pair and $r=1$,  Malle \cite[Conjecture 2.1]{Malle2010}
gives a conjecture for the distribution of $\Cl_{K^{\Gamma'}/K_0}[p^\infty]$ and we show below that in this case, our conjecture specializes to 
 \cite[Conjecture 2.1]{Malle2010}. 
\footnote{Malle also makes some conjectures  when $p\mid|\Gamma|$, but we don't discuss those here.}
 When $(\Gamma,\Gamma')$ is an absolutely irreducible pair and $r\geq 1$, Adam and Malle make a conjecture
 \cite[Conjecture 5.1]{Adam2015} that does not give formulas for their conjectured  distribution of $\Cl_{K^{\Gamma'}/K_0}[p^\infty]$, but instead conjectures that 
 the distribution of $\Cl_{K^{\Gamma'}/K_0}[p^\infty]$ agrees with
 a limit of certain distributions from finite matrix groups.  In Section~\ref{S:aip}, we show that the limit of these distributions from matrix groups exists and, indeed, agrees with our Conjecture~\ref{C} and the explicit formulas we give.
 
 Malle also makes some conjectures in the cases $\Gamma=C_3,C_5$ and $p=2$ (\cite[(1),(2)]{Malle2008},\cite[(6-1),(6-2),(6-3),(6-4)]{Malle2010}), and we check in Section~\ref{S:C35}
 that Conjecture~\ref{C} agrees with all of these conjectures in these specific cases.
 
 However, Malle makes a speculation for the situation in which $\Ind_{\Gamma'}^\Gamma \Q$ is the sum of one trivial representation and an irreducible representation (over $\Q$) \cite[Section 2]{Malle2010}, and we show in Section~\ref{S:ip} that in general his speculation does not agree with Conjecture~\ref{C}.  We show below that the only Galois cases
in which 
  Malle's speculation agrees with
 Conjecture~\ref{C} are when $\Gamma$ is $C_3$ or $C_5$ and $p=2$.  Luckily, these seem to be precisely the cases in which Malle has computational evidence for his speculation (for $p\nmid |\Gamma|$).  In many cases of this situation, our conjecture is qualitatively different from  \cite[Section 2]{Malle2010}
 because we conjecture (and prove in Theorem \ref{T:da}) that certain modules occur with probability $0$ while  \cite[Section 2]{Malle2010} assigns positive probability to each module.
  
 \subsection{Absolutely irreducible pairs}\label{S:aip}
 
Let  $(\Gamma,\Gamma')$ be an absolutely irreducible pair and $p$ a prime $p\nmid |\Gamma|$.
In the notation of Section~\ref{S:NG}, we then have $|\mathcal{M}|=1$ (without loss of generality we say $\mathcal{M}=\{2\}$)
and $d_2=1$, and $m_2=1$, and $q_2=p$.  
  We have that $\mathfrak{o}=\Z_p$
 (\cite[Proposition 8.16]{Wang2021}).
The argument in \cite[Proposition 8.16]{Wang2021} shows that 
$\mathfrak{o}=\Z_p$ if and only if $(\Gamma,\Gamma')$ is an absolutely irreducible pair.
This highlights the appeal of this condition in this context, 
for when it fails, 
the groups $\Cl_{K^{\Gamma'}/K_0}[p^\infty]$  have an $\mathfrak{o}$-module structure, for some $\mathfrak{o}$ larger than $\Z_p$,
 not just abelian group structure.  
 
Thus in the  absolutely irreducible pair situation, we are looking simply for a distribution of finite $\Z_p$-modules.
For a partition $\lambda$, which we extend by $0$ to an $(m-1)$-tuple of partitions, we have the $\mathfrak{o}$-module
 $(V_\lambda)^{\Gamma'}=\prod_i \Z/p^{\lambda_i} \Z.$  We also have $|V_\lambda|=|\prod_i \Z/p^{\lambda_i} \Z|^{n_2}$
 by Equation~\eqref{E:sizechange}
 (and $n_2=[\Gamma:\Gamma']-1$).

\begin{lemma}\label{L:ai}
Let $(\Gamma,\Gamma')$ be an absolutely irreducible pair and $p$ a prime $p\nmid |\Gamma|$, then $\Gamma$ has even order, $\sigma(2)=2$,
and $\epsilon_2=1$.
\end{lemma}
\begin{proof} Any permutation representation, where the group acts by permuting basis vectors $e_i$, may be endowed with an invariant symmetric perfect pairing giving by the formula $\langle e_i, e_j \rangle = \delta_{ij}$. It follows that the permutation representation $\Ind_{\Gamma'}^\Gamma \C$ is self-dual.  If $\Ind_{\Gamma'}^\Gamma \C=\C \times W$ for an irreducible representation $W$,
since $W$ is non-trivial, it must be self-dual.  Odd order groups have no non-trivial self-dual irreducible representations,  which implies $|\Gamma|$ is even. 
Let $W_p$ be the reduction mod $p$ of the irreducible representation $W$, i.e. $W_p=( \Ind_{\Gamma'}^\Gamma \mathbb F_p) /\mathbb F_p$. We have that  $(\Sym^2 \Ind_{\Gamma'}^\Gamma \F_p)^{\Gamma}$ is at least two dimensional, since it includes the pairing on 
$\Ind_{\Gamma'}^\Gamma \F_p$ that first takes the quotient to the trivial representation, as well as the perfect pairing mentioned above, which implies
that $(\Sym^2 W_p)^{\Gamma}\ne 0$.   Since $p$ is odd, this implies $W_p$ is self-dual (i.e. $\sigma(2)=2$) and $\epsilon_2=1$.
\end{proof}

\subsubsection{$r=1$}
First,we  assume $r=1$.
We then see that in this case (absolutely irreducible pair and $r=1$) 
the formula for the measure $\nu_{\Gamma'}$ of Proposition~\ref{P:NG} simplifies significantly.
We have, for a partition $\lambda$,  and group $U= \prod_i \Z/p^{\lambda'_i} \Z$, our conjecture for the distribution of $\Cl_{K^{\Gamma'}/K_0}[p^\infty]$  is given by
\begin{align}\label{E:ai}
\nu_{\Gamma'}(\{ U\})=
\frac{p^{\binom{\lambda_1}{2}}
}{|U|^{un_2}|\Aut(U)|}
\prod_{\ell\geq 0}(1+ p^{-un_2-1-\ell})^{-1} 
\prod_{\ell=  1}^{  \lambda_{1} } (1 - p^{ - un_2- \ell } ).
\end{align}
Now we compare this to Malle's \cite[Conjecture 2.1]{Malle2010}.
(We assume that  \cite[Conjecture 2.1]{Malle2010} has a typo and the product in the numerator should start at $i=u+1$ since
this modification is what is shown in the paper of Adam and Malle \cite[Example 5.3]{Malle2010}, where it is said to be exactly 
 \cite[Conjecture 2.1]{Malle2010}.)  We need to compute the $u$ in Malle's formula (which is not quite our $u$).
 Since we have that $K/K_0$ is split at all infinite places by Remark~\ref{R:scatinf}, we have that Malle's $\chi_E$ is the character that is
 the sum of a copy of the character of the regular representation for each infinite place of $K_0$ minus one trivial character.  It follows that Malle's $u$ is our $un_2$ (in the absolutely irreducible pair case).  Then we see that Equation~\eqref{E:ai} exactly agrees with \cite[Conjecture 2.1]{Malle2010} (with the typo corrected).
 
 Even though both these conjectures imply a conjecture on $p$-ranks, 
as a double check,  
  we compare our conjectured  formula for the distribution of $p$-ranks to Malle's \cite[Proposition 2.2]{Malle2010}.
  To obtain our formula, we apply Proposition~\ref{P:NG} with $k=1$ (leading us to Theorem~\ref{T:formula} with $\underline{k}=(1,0,0,\dots)$.
\begin{align}\label{E:ai2}
\nu_{\Gamma'}(\{ U | \rk_p U =\lambda \})=
\frac{p^{\binom{\lambda}{2}}
}{|(\Z/p\Z)^\lambda|^{un_2}|\Aut((\Z/p\Z)^\lambda)|}
\prod_{\ell\geq 0}(1+ p^{-un_2-1-\ell})^{-1} .
\end{align} 
This agrees with the distribution of \cite[Proposition 2.2]{Malle2010}.

\subsubsection{$r\geq 1$} 
 Now we consider a general $r\geq 1$ (still with an absolutely irreducible pair $(\Gamma,\Gamma')$ and $p\nmid |\Gamma|$).  
In this case, Adam and Malle \cite[Conjecture 5.1]{Adam2015} conjecture that $\Cl_{K^{\Gamma'}/K_0}[p^\infty]$ is distributed as the limit of certain distributions coming from matrix groups, and we describe those distributions here.  Given a prime $p$, and positive integers $r\leq k,g$, let 
$\Sp^{(r)}_{2g}(\Z/p^k\Z)$ be the subgroup of matrices of $ \GL_{2g}(\Z/p^k\Z) $, that when reduced mod $p^r$ are in
$\Sp_{2g}(\Z/p^r\Z)$.  In other words, these are matrices that are symplectic mod $p^r$.  Adam and Malle define a distribution
$\nu^r_{AM}$ of finite abelian $p$-groups such that
\begin{equation}\label{E:AM}
\nu^r_{AM}(\{G\}):=\lim_{g\ra\infty} \frac{|\{A\in \Sp^{(r)}_{2g}(\Z/p^k\Z) \,|\, \ker(A-I_{2g})\isom G |}{|\Sp^{(r)}_{2g}(\Z/p^k\Z)|},
\end{equation}
where $I_{2g}$ is the $2g\times 2g$ identity matrix, and we take the kernel of $A-I_{2g}$ as a map $(\Z/p^k\Z)^{2g} \ra (\Z/p^k\Z)^{2g}$ and $k$
is any positive integer such that $p^{k-1}G=0$.
As far as we can tell, they do not show that these limits exist or that they give a probability distribution, but we will show that here.
\begin{lemma}
For $r\geq 1$,  the limits defining $\nu^r_{AM}$ exist and define a probability measure on the set of (isomorphism classes of) finite abelian $p$-groups
 such that $\int_G |\Sur(G,H)|   d\nu^r_{AM}=|\wedge^2 H[p^r]|$ for every finite abelian $p$-group $H$. 
\end{lemma}
\begin{proof}
First, we note that by Proposition~\ref{P:DVRselfdual} there is a probability measure  $\nu^r$
on the set of finite abelian $p$-groups
with these moments, and the moments are well-behaved.
For $B$ a $2g\times 2g$ matrix over $\Z/p^k\Z$, we have $\ker(B)\isom \cok(B)$, and so we can replace 
$\ker(A-I_{2g})$ with $\coker(A-I_{2g})$. 
When $p^{k-1}G=0$,  the cokernel of a $\Z_p$ linear map $B :\Z_p^{2g} \ra \Z_p^{2g}$
 is isomorphic to $G$ if and only if the same is true for the reduction of $\bar{B}: (\Z/p^k\Z)^{2g} \ra (\Z/p^k\Z)^{2g}$ mod $p^k$.
So we replace $\Z/p^k\Z$ with $\Z_p$ (using the analogous definition for 
 $\Sp^{(r)}_{2g}(\Z_p)$),  and counting measure on  $\Sp^{(r)}_{2g}(\Z/p^k\Z)$ with Haar measure on
  $\Sp^{(r)}_{2g}(\Z_p)$.
 
Let $\nu_g$ be the distribution of $\coker(A-I_{2g})$ for a Haar random $A\in \Sp^{(r)}_{2g}(\Z_p)$.  
By Theorem \ref{T:mom},  if we show that 
\begin{equation}\label{E:getlimmom}
\lim_{g\ra\infty} 
\int_G |\Sur(G,H)|   d\nu_{g}=|\wedge^2 H[p^r]|,
\end{equation}
for every finite abelian $p$-group $H$, then the $\nu_g$ weakly converge to $\nu^r$, and since the delta function on a single group is continuous,
the limits defining $\nu^r_{AM}(\{G\})$ exist and we have $\nu^r_{AM}(\{G\})=\nu^r(\{G\})$, which would complete the proof of the lemma.

To show \eqref{E:getlimmom}, we have that  $\Sp^{(r)}_{2g}(\Z_p)$ acts on $\Sur(\Z_p^{2g},H)$, and the fixed points of 
$A\in \Sp^{(r)}_{2g}(\Z_p)$ are precisely
the elements of 
$\Sur(\coker(A-I_{2g}),H)$.  By Burnside's Lemma, we then have that the average size of $\Sur(\coker(A-I_{2g}),H)$ over $A\in \Sur(\Z_p^{2g},H)$ is the
number of orbits of $\Sp^{(r)}_{2g}(\Z_p)$ on $\Sur(\Z_p^{2g},H)$.
By \cite[Theorem 2.14]{Michael2006}, we have that the orbits of $\Sp_{2g}(\Z_p)$ on $\Sur(\Z_p^{2g},H)$, for $g$ sufficiently large given
$H$, are precisely the fibers of a map
$$
F: \Sur(\Z_p^{2g},H) \ra \wedge^2 H
$$
that we now describe.  Let $x_1,\dots,x_{2g}$ be a standard basis of $\Z_p^{2g}$, and then $F(\phi)=\sum_{i=1}^{g} \phi(x_i) \wedge \phi(x_{i+g})$.
The same statement for $H/p^r H$, tells us that 
$$
\bar{F}: \Sur(\Z_p^{2g},H) \ra \wedge^2 (H/p^rH)
$$
is invariant on $\Sp^{(r)}_{2g}(\Z_p)$ orbits.   

Since $| \wedge^2 (H/p^rH)|=|\wedge^2 (H) [p^r]|$, we will be done if we can show
that each fiber of $\bar{F}$ is a single $\Sp^{(r)}_{2g}(\Z_p)$ orbit.  
Consider two elements $a,b$ in the same fiber of $\bar{F}$.  
We pick a minimal generating set $h_1,\dots,h_s$ of $H$.
We write $F(a)=\sum_{i<j} a_{i,j} h_i\wedge h_j$
 and
$F(b)=\sum_{i<j} b_{i,j} h_i\wedge h_j$ for $a_{i,j},b_{i,j}\in \Z_p$ such that $a_{i,j} \equiv b_{i,j} \pmod{p^r}$ for all $i,j$.
We chose a bijection $u : \{i \,|\, 1\leq i \leq \binom{s}{2} \} \ra \{(i,j) \,|\, 1\leq i<j \leq s \} $ and write $u_1(i),u_2(i)$ for the two components.
For $g$ large enough, given $H$ (as above and so $g\geq s+\binom{s}{2}$), 
we can choose a surjection  $ \phi_a \in \Sur(\Z_p^{2g},H)$ such that 
\begin{itemize}
\item $\phi_a(x_i)=h_i$ for $1\leq i \leq s$,
\item $\phi_a(x_{s+i})= a_{u(i)} h_{u_1(i)}$ for $1\leq i \leq \binom{s}{2}$,
\item $\phi_a(x_i)=0$ for $s+\binom{s}{2}<i \leq g$, 
\item $\phi_a(x_{g+i})=0$ for $1\leq i \leq s$ and $s+\binom{s}{2}<i \leq g$,
and
\item $\phi_a(x_{g+s+i})=  h_{u_2(i)}$ for $1\leq i \leq \binom{s}{2}$.
\end{itemize}
We have $F(\phi_a)=F(a)$, and so $\phi_a$ and $a$ are in the same
$\Sp_{2g}(\Z_p)$ orbit and hence the same $\Sp^{(r)}_{2g}(\Z_p)$ orbit.
We choose $\phi_b$ similarly.  
There is a $\Z_p$-module map sending $x_{s+i}\mapsto x_{s+i} +(b_{u(i)}-a_{u(i)})x_{u_1(i)}$ for $1\leq i \leq \binom{s}{2}$
and mapping all other $x_i\mapsto x_i$.  This map  is the identity mod $p^r$ and takes $\phi_a$ to $\phi_b$.
Thus we conclude $\phi_a$ and $\phi_b$ (and hence $a$ and $b$) are in the same $\Sp^{(r)}_{2g}(\Z_p)$ orbit, which concludes the proof.
\end{proof}

Proposition~\ref{P:DVRselfdual} gives explicit formulas for $\nu^r_{AM}$, and since we apply the proposition in the 
case that $\epsilon=s=0$, we even have nice product formulas for this measure using Lemma~\ref{R:e0}, for any $r\geq 1$.
Adam and Malle noted it was possible to calculate $\nu^r_{AM}(G)$ explicitly for any given $r$ and $G$ \cite[Remark 3.5]{Adam2015}, but did not have a closed formula for $r\geq 3$.

Finally, we are ready to state Adam and Malle's conjecture \cite[Conjecture 5.1]{Adam2015}.  
For  an absolutely irreducible pair $(\Gamma,\Gamma')$,  Adam and Malle \cite[Conjecture 5.1]{Adam2015} conjecture that $\Cl_{K^{\Gamma'}/K_0}[p^\infty]$ is distributed like the quotient of a random group from
$\nu^r_{AM}(G)$ by $un_2$ uniform, independent random elements (as above Malle's $u$ is our $un_2$).
Such a random group has $H$-moment $ |\wedge^2 H[p^r]| |H|^{-un_2}$ for every $H$, and by Proposition~\ref{P:DVRselfdual} there is a unique distribution with these moments.   By Proposition~\ref{P:NG} and Lemma~\ref{L:selfdualmom}, we see that Conjecture~\ref{C}
predicts the distribution with the same moments, and so our conjectures agree in this case.

In the special case when $\Gamma=\Z/2\Z$, there is also a conjecture \cite[Conjecture 1.2]{Lipnowski2020} which gives a distribution for the class group together with extra invariants $\omega, \psi$. Ignoring the extra invariants, \cite[Conjecture 1.2]{Lipnowski2020} implies a conjecture simply for the distribution of the class group. This agrees with Conjecture~\ref{C}. To check this, it suffices to check that the moments of the distribution $Q^t\mu$ of \cite[Conjecture 1.2]{Lipnowski2020} agree with the moments described in Conjecture~\ref{C}, since we have seen that these moments uniquely determine a distribution. The $(G, \omega_G,\psi_G)$-moment of $Q^t\mu$ is calculated in \cite[Lemma 8.12 and Theorem 2.3]{Lipnowski2020} as $\frac{1}{ |G|^t | \Sym^2 (G)[\ell^n] |}$ for each possible value of the invariants $\omega_G ,\psi_G$. It is straightforward to check by writing $G$ as a product of cyclic groups that there are $ | \Sym^2 (G)[\ell^n] |$ possible values of $\psi_G$ for each $\omega_G$ and $| \wedge^2(G) [\ell^n]|$ possible values of $\omega_G$, so summing over these possibilities, the $G$-moment is $\frac{ | \wedge^2(G)[\ell^n]|}{|G|^t}$. This agrees with Conjecture~\ref{C} since $t$ and $u$ are both defined to equal the number of infinite places.

\subsection{$C_3,C_5$}\label{S:C35}

Now let $\Gamma$ be $C_\ell$ for $\ell=3,5$ (and $\Gamma'=1$) and $p=2$. 
 In this case,  there is only one non-trivial indecomposable module
$V_2$.
We have $n_2=1$ and $d_2=\ell-1$ and $\epsilon_2=0$.  
When $r=1$, our conjecture for the distribution of $\Cl_{K^{\Gamma'}/K_0}[p^\infty]$
(from Conjecture~\ref{C}, using Corollary~\ref{C:specific}, Lemma~\ref{L:selfdualmom} for the formulas) is
\begin{align*}
&\nu(\{ V_{\underline{\lambda}'} \})=
\frac{2^{\frac{\ell-1}{2}\lambda_1^2}}{|V_{\underline{\lambda}'}|^u|\Aut(V_{\underline{\lambda}'})|}
\prod_{i\geq 0}(1+ (2^{\ell-1})^{-u-\frac{1}{2}-i})^{-1} 
\prod_{i=  1}^{  \lambda_{1} } (1 - (2^{\ell-1})^{ - u- i } ). 
\end{align*}
Note that 
$$
\prod_{i\geq 0}(1+ (x^2)^{-\frac{1}{2}-i})=\prod_{i\geq 0}(1+ x^{-2i-1})=\frac{\prod_{i\geq 0}(1-(x^2)^{-2i-1})}{\prod_{i\geq 0}(1- x^{-2i-1})}
=\frac{\prod_{i\geq 0}(1-(x^2)^{-i})}{\prod_{i\geq 0}(1-(x^2)^{-2i})}\cdot\frac{\prod_{i\geq 0}(1- x^{-2i})}{\prod_{i\geq 0}(1- x^{-i})}.
$$
From this one can check that Malle's conjectures \cite[(6-2),(6-4)]{Malle2010} in the cases $u=1,2$ agree with our conjecture for $C_3$  and $r=1$.
It then formally follows that Malle's conjectures \cite[(6-1),(6-3)]{Malle2010} for the $2$-ranks agree with ours (and it is also  easily verified using 
Theorem~\ref{T:ptors}).
One can also check that Malle's conjecture \cite[Equation (2)]{Malle2008} agrees with our conjecture for $C_5$ and $u=r=1$. 
Malle \cite[Section 6.4]{Malle2010} suggests that the $2$-rank prediction for $u=1$
\cite[(6-1)]{Malle2010} 
should still hold for $r>1$,
which agrees with our conjecture (see Remark~\ref{R:rind}).

 \subsection{Irreducible pairs}\label{S:ip}
We say $(\Gamma,\Gamma')$ is an \emph{ irreducible pair} if the representation $\Ind_{\Gamma'}^\Gamma \Q$ is the sum of one trivial representation and an irreducible representation $W$ (over $\Q$), and moreover that $W$ has Schur index $1$ (so $\End_\Gamma(W)$ is a field and not just a division algebra).

\subsubsection{Irreducible over $\Q_p$}
Let $(\Gamma,\Gamma')$ be an irreducible pair.
We first consider the case when $W_p:=W\tensor_{\Q} \Q_p$ is irreducible over $\Q_p$.
In the notation of Section~\ref{S:NG}, we then have $|\mathcal{M}|=1$ (without loss of generality we say $\mathcal{M}=\{2\}$).
Let $d=d_2$ and $q=q_2=p^d$.  We have $m_2=1$.
As in the proof of Lemma~\ref{L:ai}, we have that $W_p$ is self-dual.
By \cite[Section 8.3]{Wang2021}, we have that $\mathfrak{o}$ is the maximal order in the field $\End_{\Gamma} W_p$.
Then our conjecture for the distribution of $\Cl_{K^{\Gamma'}/K_0}[p^\infty]$, in the case $r=1$, is given by,
 for $U=V_{\underline{\lambda}'}^{\Gamma'}$,
\begin{align}
\nu_{\Gamma'}(\{U \})&= \notag
\frac{|(\wedge^2_{\Z_p} V_{\underline{\lambda}'})^\Gamma[p]|}{|V_{\underline{\lambda}'}|^u|\Aut_{\mathfrak{o}}(U)|}
\prod_{\ell\geq 0}(1+ q^{-un_2-\frac{\epsilon_2+1}{2}-\ell})^{-1} 
\prod_{\ell=  1}^{  \lambda_{1} } (1 - q^{ - un_2- \ell} )\\
&=\frac{
q^{
\binom{\lambda_1}{2} 
+(1-\epsilon_2)\frac{\lambda_1}{2} }
}{|U|^{n_2u}|\Aut_{\mathfrak{o}}(U)|}
\prod_{\ell\geq 0}(1+ q^{-un_2-\frac{\epsilon_2+1}{2}-\ell})^{-1} 
\prod_{\ell=  1}^{  \lambda_{1} } (1 - q^{ - un_2- \ell } ). \label{E:ourspec}
\end{align}
using Proposition~\ref{P:NG} and Theorem~\ref{T:formula} for the first line, and Lemma~\ref{L:selfdualmom} and \eqref{E:sizechange} for the second line.
In \cite[Section 2]{Malle2010}, for $r=1$ Malle suggests that the probability for $U$ should be 
\begin{equation}\label{E:Mallespec}
c' 
d^{\lambda_1}
\frac{q^{\binom{\lambda_1}{2}} }{|U|^{n_2u}|\Aut_{\mathfrak{o}}(U)|}
 \prod_{\ell=1}^{\lambda_1} (1-   q^{-un_2-\ell}),
\end{equation}
for some unspecified $c'$ not depending on $U$. (We have absorbed some  other constant parts of the formula into
the unspecified constant from \cite{Malle2010}. Also, one can check that given that $W_p$ is irreducible,
our $\mathfrak{o}$ is $\mathcal{O}\tensor_\Z \Z_p$ for Malle's $\mathcal{O}$, and so our $d$ agrees with Malle's $d$.)
In order for \eqref{E:ourspec} and \eqref{E:Mallespec} to agree, we must have
$$
p^{\frac{1-\epsilon_2}{2} d }=d.
$$
Since $d$ is a positive integer, and $p$ is a prime, the above equality holds (i.e. our conjectures agree) exactly when
\begin{itemize}
\item $\epsilon_2=1$ and $d=1$ (so $W_p$ is absolutely irreducible over $\Q_p$ and hence $\Q$ and we have an absolutely irreducible pair, as discussed above), or
\item $\epsilon_2=0$ and $p=2$ and $d=2,4$.
\end{itemize}
If we are interested in the Galois case (so $\Gamma'=1$), then $d=|\Gamma|-1$, and the only possibilities where $d=2,4$ are the $\Gamma=C_3,C_5$ examples discussed above.

\subsubsection{Reducible over $\Q_p$}
Let $(\Gamma,\Gamma')$ be an irreducible pair.
If $W_p$ is reducible over $\Q_p$, then that means $p$ splits in the field $\End(W)$.  Let $\mathcal{O}$ be the maximal order of $\End(W)$.
Malle \cite[Section 2]{Malle2010} gives a suggestion in terms of the ``$\mathcal{O}$-rank'' of the $p$-group $H$.  Since $\O/p$ is not a field, it is not clear to the authors what this rank means.  Let us consider $\Gamma=C_7$ and $\Gamma'=1$ and $p=2$.  
We have $\mathcal{O}=\mathfrak{o}=\Z[\zeta_7]$, but $\mathcal{O}\tensor_\Z \Z_2=R_1\times R_2$ for two extensions $R_i$ of degree $3$ over $\Z_2$
(both isomorphic to the same unramified extension).   
We have $q=q_2=2^3$.
In the case $r=1$, our conjecture now gives probabilities
\begin{align*}
&\nu(\{ V_{\underline{\lambda}'} \})=
\frac{q^{\lambda^2_1\lambda_1^3}}{|V_{\underline{\lambda}'}|^u|\Aut(V_{\underline{\lambda}'})|}
\frac{\eta_{q}(\infty)\eta_{q}(2u)}{\eta_{q}(u-\lambda^2_{1}+\lambda^{3}_{1})\eta_{q}(u+\lambda^2_{1}-\lambda^{3}_{1})}
 \prod_{\ell= 1}^{  \lambda^2_{1} } (1 - q^{ - u- \ell } )
\prod_{\ell=  1}^{  \lambda^{3}_{1} } (1 - q^{ - u- \ell } ),
\end{align*}
if  
 $|\lambda^2_{1} -\lambda^{3}_{1}|\leq u$ 
and $\nu(\{ V_{\underline{\lambda}'} \})=0$ otherwise.
We consider $4$ $2$-torsion $\mathcal{O}$-modules: $1$, $R_1/2$, $R_2/2$, $R_1/2\times R_2/2$.
These correspond to the tuples of partitions $(0,0),((1),0)(0,(1)),((1),(1))$.  
If we look at the quantities   $\nu(\{ V\}) |V_{}|^u|\Aut(V_{})|$ for our four modules, they are in ratios 
$$
1 : 1-2^{-3u}
: 1-2^{-3u}
: 2^3(1-2^{3(-u-1)})^2.
$$
while in \cite[Section 2]{Malle2010}, the corresponding ratio between a rank $0$ group and a rank $1$ group is
$$
1: 6(1-2^{6(-u-1)}).
$$
We do not see a way to interpret ``$\mathcal{O}$-rank'' as to make our conjecture align with Malle's suggestion.  

More qualitatively,
our conjecture assigns  probability  $0$ to any $V_{\underline{\lambda}'} $ with  $|\lambda^2_{1} -\lambda^{3}_{1}|> u$, e.g. if $u=1$,
we assign probability $0$ to any group which is rank $0$ for $R_1$ and rank $2$ for $R_2$.  In contrast,
\cite[Section 2]{Malle2010} suggests that every $2$-group $\O$-module occurs with positive probability.
So in this setting our Conjecture~\ref{C}, and Theorem \ref{T:da}, do not agree with Malle's suggestion.

\end{document}